\documentclass[10pt]{amsart}
\usepackage[english]{babel}
\usepackage{graphicx}
\usepackage[numbers]{natbib}
\usepackage[colorlinks=true, allcolors=blue]{hyperref}
\usepackage{amsmath,amsthm,amssymb,amsfonts}
\usepackage{mathrsfs,xfrac,nicefrac,faktor,esint,dsfont}
\usepackage{enumitem}
\usepackage{tikz-cd}
\usepackage{accents}
\usepackage{xparse}
\usepackage{hyperref}

\allowdisplaybreaks

\usepackage{tikz} 
\usetikzlibrary{calc}
\usetikzlibrary{intersections}
\usetikzlibrary{patterns}

\title[NONUNIQUENESS of sMHD]{Non-uniqueness of weak solutions to the 3-D stationary MHD equations in Besov space with negative regularity index} 
\author{Peng Qu}
\address{School of Mathematical Sciences \& Shanghai Key Laboratory for Contemporary Applied Mathematics, Fudan University, China.}
\email{pqu@fudan.edu.cn}

\author{Mingxin Zhang}
\address{School of Mathematical Sciences, Fudan University, Shanghai 200433, China.}
\email{mxzhang24@m.fudan.edu.cn}

\date{\today} 

\date{\today}

\makeatother

\numberwithin{equation}{section}

\newtheorem{theorem}{Theorem}[section]
\newtheorem{definition}[theorem]{Definition}

\newtheorem{lemma}[theorem]{Lemma}
\newtheorem{remark}[theorem]{Remark}

\numberwithin{equation}{section}

\newcommand*{\Id}{\ensuremath{\mathrm{Id}}}
\newcommand*{\supp}{\ensuremath{\mathrm{supp\,}}}

\newcommand{\aint}{{\fint}}

\newcommand{\T}{{\mathbb{T}}}

\newcommand{\ru}{\mathring{R}_{q}^u}
\newcommand{\rb}{\mathring{R}_{q}^B}

\renewcommand{\P}{P}

\renewcommand{\div}{{\mathrm{div}}}
\newcommand{\curl}{{\mathrm{curl}}}

\renewcommand{\d}{{\rm d}}

\newcommand{\norm}[1]{\lVert#1\rVert}

\newcommand{\la}{\lambda_q}

\newcommand{\rs}{r}
\newcommand{\xs}{\xi}

\newcommand{\rotimes}{\mathring{\otimes}}
\newcommand{\Hs}{{H}^{-s}}

\newcommand{\R}{{\mathbb R}}
\newcommand{\lbb}{\overline{\lambda}}

\def\lbb{\lambda}

\def\9{{\infty}}
\def\ve{{\varepsilon}}
\def\na{{\nabla}}

\def\({\left(}
\def\){\right)}

\begin{document}

\maketitle 

\begin{abstract}
\noindent
In this paper, we study the 3-D incompressible fractional stationary MHD equations on the torus $\T^3$.
For fractional power indices $\alpha_1,\alpha_2 >0,$ we prove that there exist infinite non-trivial stationary singular solutions of 3-D  fractional MHD equations via convex integration.
In particular, this result shows the non-uniqueness of weak solutions to the 3-D stationary MHD equations in some sub-critical Besov spaces.

\end{abstract}

\section{Introduction and main results}
\subsection{Introduction}

We would consider the stationary 3-D incompressible fractional MHD equations on the torus $\T^3:=[-\pi,\pi]^3$:
\begin{equation}\label{equa-sMHD}
	\left\{\aligned
	& \nu_1 (-\Delta)^{\alpha_1} u+ \div(u\otimes u - B \otimes B ) + \nabla P=0,  \\
	& \nu_2 (-\Delta)^{\alpha_2} B+ \div(B\otimes u - u \otimes B ) =0, \\
	& \div u = 0, \quad \div B = 0,\\
    & \hat{u}(0)= 0, \quad \hat{B}(0)= B_{*},
	\endaligned
	\right.
\end{equation}
where $u=(u_1,u_2,u_3)^\top(x)\in \R^3$ denotes the velocity field, 
$B=(B_1,B_2,B_3)^\top(x)\in \R^3 $ denotes the magnetic field, 
$ P=P(x)\in \R $ denotes the fluid pressure. 
Moreover, $\alpha_1$, $\alpha_2 \in (0 , +\infty)$ and $(-\Delta)^\alpha$ denotes the fractional Laplacian operator defined on distributions by the Fourier transform
\[
\mathscr{F}((-\Delta)^{\alpha}u)(k) = |k|^{2\alpha}\mathscr{F}(u)(k), \quad k \in \mathbb{Z}^3.
\]
The constants $\nu_1,\nu_2 > 0$ correspond to the viscosity and resistivity coefficients, respectively.
For the most important case $\alpha_1 = \alpha_2 = 1$, one knows the following critical spaces' chain:
\[ \dot{H}^{\frac{1}{2}} \hookrightarrow L^3 \hookrightarrow \dot{B}^{-1+\frac{3}{q}}_{q,r} \hookrightarrow \mathrm{BMO}^{-1} \hookrightarrow \dot{B}^{-1}_{\infty,\infty}, \quad 3 < q < \infty. \]
And for general $\alpha_1 = \alpha_2 = \tilde{\alpha}$, we can check that the  homogeneous Besov space $\dot{B}^{-\theta}_{q,r}(\mathbb{R}^3)$ is critical when the exponent satisfies
$$\theta = 2\tilde{\alpha} - 1 - \frac{3}{q}.$$
For $\tilde{\alpha} > 1/2$, there exist critical and sub-critical Besov spaces with negative regularity index.
And in this paper, we would show that the weak solutions are non-unique in the Besov spaces with negative indices, even if the space 
is critical or sub-critical.

The MHD equations play an important role in understanding the interaction between electrical fluids and magnetic fields, with many applications across various physial scenarios.
In the development of the corresponding mathematical theory, M. Sermange and R. Temam \cite{ST83} established the local well-posedness of the MHD equations in both two and three dimensions for smooth initial data 
and J. Wu \cite{wu03} proved the global weak solutions for $L^2$ initial data when the viscosity and resistivity are positive 
and global smooth solutions exist for initial data in $H^{5/2}$ when the viscosity and resistivity index $\alpha$ are higher. 
More results on well-posedness and conservative properties of MHD equations can see \cite{CKS97,fl20,FL22,KL07,taylor86}.
In the past decade, it have made much progress on non-uniqueness of MHD equations. R. Beekie, T. Buckmaster, and V. Vicol \cite{bbv20} established the existence of weak solutions to the ideal MHD equations that fail to conserve magnetic helicity.
M. Dai \cite{dai18} proved non-uniqueness of Leray-Hopf weak solutions of the 3D Hall-MHD system. 
Y. Li, Z. Zeng, and D. Zhang \cite{lqzz22} proved the non-uniqueness of weak solutions to the three-dimensional hyperviscous MHD equations.  
C. Miao and W. Ye \cite{MY24} proved the non-uniqueness of weak solution to the viscous and resistive MHD system in $ C([0,T]; L^2(\mathbb{T}^3))$ if the solutions have at least one interval of regularity. 
C. Miao, Y. Nie and W. Ye \cite{MNY25} proved Onsager-type conjecture for the Els\"asser energies of the ideal MHD equations. 
At the same time,  Y. Nie and W. Ye \cite{NY2025} proved a sharp and strong nonuniqueness for a class of weak solutions to the three-dimensional magnetohydrodynamic system. 
A. Enciso, J. {Pe{\~n}afiel-Tom{\'a}s} and D. Peralta-Salas \cite{EPT26} constructed H\"older continuous solutions to ideal MHD with nonzero helicity that do not preserve the Els\"asser energies. 
Many of these non-uniqueness results are based on the method of convex integration.

Thanks to the breakthroughs of C. De Lellis and L. Sz\'{e}kelyhidi \cite{dls09, DS13}, remarkable advances have been achieved in the study of non-uniqueness of weak solutions to syetems of fuild dynamics in the past two decades.
Inspired by J. Nash's $C^1$ isometric embedding theorem \cite{Nash1954}, they developed the convex integration method and proved the existence of wild solutions for the 3-D Euler equations.
Following this, Isett \cite{I18} proved the flexibility part of Onsager's conjecture for the three-dimensional Euler equations in the class $C_{t,x}^{1/3-}$.
Building on their work, significant results have been made, for instance \cite{bdis15, nv23}.

Another major breakthrough came in 2019 when T. Buckmaster and V. Vicol \cite{bv19b} established the non-uniqueness of solutions to the 3-D Navier-Stokes equations in the space $C_tL^2_x$.
Their work introduced the method of intermittent convex integration. 
A. Cheskidov and X. Luo \cite{cl20.2} established sharp non-uniqueness results for the incompressible Navier-Stokes equations for any dimension $d \geq 2$, by introducing temporal intermittency. 
These methods have been successfully applied to other models.
For instance, the hyperviscous 3-D Navier-Stokes equations \cite{lqzz24,lt20}, the low viscosity 2-D incompressible equations \cite{lq20},  
the transport equations \cite{cl21,cl22,ms18}, and the stationary Navier-Stokes equations \cite{luo19}.

Most of the non-uniqueness results are obtained in super-critical spaces with negative regularity index. 
M. Christ \cite{CM05} proved that solutions of the Cauchy problem for the Navier-Stokes equation are not unique in $C^0([0,1],H^s(\mathbb{T}^2))$ for each $s < 0$.
Also, M. Christ \cite{CS05} proved that Generalized solutions of the Cauchy problem for the one-dimensional periodic nonlinear Schrodinger equation with cubic or quadratic nonlinearities are not unique $C^0([0,1],H^s(\mathbb{T}^2))$ for each $s < 0$. 
Recently, M. Coiculescu and S. Palasek \cite{MS26} demonstrated the non-uniqueness of solutions to the 3D Navier-Stokes equations in the critical space $ \mathrm{BMO}^{-1} $. 
E. Ashkarian, A. Bhargava, N. Gismondi and M. Novack \cite{ABGM25} proved that there exist nontrivial weak solutions to stationary 2D Navier-Stokes equations belonging to $\displaystyle\bigcap_{\varepsilon\in(0,1)} L^{2-\varepsilon}(\mathbb{T}^2)\cap \dot{H}^{-\varepsilon}(\mathbb{T}^2)$. 
And N. Gismondi and A. F. Radu \cite{NA25} proved that the non-uniqueness of non-trivial solutions to the stationary dissipative surface quasi-geostrophic equation on the two dimensional torus which lie strictly below the critical regularity threshold of $\dot{H}^{-1/2}(\mathbb{T}^2)$. 
M. Fujii \cite{MI26} established the non-uniqueness of solutions to the Navier-Stokes equations in the critical space ${B}_{q,r}^{-1+\frac{d}{q}}$ for any dimensions $d \geq 2$. 
Furthermore, A. Cheskidov and H. Hou \cite{CH26} proved that non-uniqueness holds for the Navier-Stokes equations in any space ${B}_{q,r}^{-\theta}\ (\theta > 0)$ for any dimensions $d \geq 2$ which includes the sub-critical cases.
Inspired by \cite{ABGM25} \cite{CH26}  and \cite{MI26}, in this paper we intend to prove the non-uniqueness of the stationary MHD equations in Besov spaces ${B}_{q,r}^{-\theta}\ (\theta > 0), $ espically including those in the subcritical regime.

\subsection{Main results}
In this section, we aim to state Theorem \ref{thm2}, which are the main non-uniqueness result in this paper. 
First, we recall the definition of Besov spaces.

For any distribution $u \in \mathscr{D}'(\mathbb{T}^3)$, let $\hat{u}$ represent its Fourier transform, which is an element of $\mathcal{S}'(\mathbb{Z}^3)$ defined through the duality pairing
$$\hat{u}(k) = \langle u, e^{-2\pi i k \cdot x} \rangle, \quad k \in \mathbb{Z}^3,$$ 
where the bracket $\langle\cdot,\cdot\rangle$ signifies the action of distributions on smooth periodic functions.

Define $\chi \in C^{\infty}_c(\mathbb{R}^3)$ be a radial bump function that is supported in $B(0,1)$ and equals 1 on $B(0,9/10)$. 
For any dyadic integer $K \ge 1$, we can define the \emph{Littlewood-Paley projection $P_{\le K}$} by
\begin{align}
	\label{Pkoper1}
	 P_{\le K} (u)(x) := \sum_{k \in \mathbb{Z}^3} \chi(k/K) \hat{u}(k) e^{2\pi i k \cdot x}, 
\end{align}
and the companion projections
\begin{align}
	\label{Pkoper2}
P_{>K} := \operatorname{Id} - P_{\le K},\quad P_1 := P_{\le 1} ,\quad P_K := P_{\le K} - P_{\le K/2} ~ \text{for } K \ge 2. 
\end{align}
Let $s \in \mathbb{R}$ and $q,r \in [1,\infty]$. The \emph{Besov space} $B^s_{q,r}$ is defined by 
\[ \|f\|_{B^s_{q,r}} := \left( \sum_{K} K^{sr} \|P_K f\|_{L^q}^r \right)^{1/r} < \infty, \] with the standard modification for the case $r=\infty$
where $f \in \mathscr{D}'(\mathbb{T}^3)$.

Since the solutions have low regularity, it is neccessary to define the tensor product in the space $H^{-s}$. 
Recall the Sobolev space $\Hs$:
\[ \|u\|_{\Hs} := \left( \sum_{k \in \mathbb{Z}^3} |k|^{{-2s}} |\hat{u}(k)|^2 \right)^{1/2} < \infty. \]
A bit different from \cite{CH26}, we define 
\begin{equation}
    \label{e:u*v}
    u \otimes v := \sum_{\tau_1, \tau_2} P_{\tau_1} u \otimes P_{\tau_2} v \quad \text{ in } \Hs.
\end{equation}
 where $u\in \mathscr{D}'(\mathbb{T}^3)$, $ s>0$, $\tau_i\ (i=1,2)$ are dyadic integers and $\displaystyle\sum_{\tau_1, \tau_2}$ denotes summation among all dyadic integer pairs.
 If 
 \begin{equation}
    \label{e:u*u}
    \sum_{\tau_1, \tau_2}\| P_{\tau_1} u \otimes P_{\tau_2} u \| _{\Hs }< \infty,
\end{equation}
we can say that $u \otimes u$ is well-defined as a paraproduct in $\Hs$.

Our strategy is to construct non-trivial weak singular solutions to this stationary system via convex integration. 
First, we give the definition of the weak solutions.
\begin{definition} [Weak solution]
For the cases $s>\max\{\alpha_1,\alpha_2\}+1/2$, we say that $(u, B)$ is a weak solution to the stationary MHD equations \eqref{equa-sMHD} if

\begin{itemize}
    \item $\hat{u}(0)= 0, \hat{B}(0)= B_{*};$	
   \item Both $u(x)$ and $B(x)$ are divergence free in the sense of distributions;
   \item $u \otimes u,B\otimes B,u \otimes B,B \otimes u$ are well-defined in $\Hs$;
   \item Equation \eqref{equa-sMHD} holds in the sense of distributions, i.e., for any divergence-free test functions $\varphi \in C^\infty(\mathbb{T}^3)$,
   \[
	-\langle u \otimes u - B \otimes B, \nabla\varphi\rangle + \langle u, \nu_1(-\Delta)^{\alpha_1}\varphi\rangle = 0,
	\]
	\[
	-\langle B \otimes u - u \otimes B, \nabla\varphi\rangle + \langle B, \nu_2(-\Delta)^{\alpha_2}\varphi\rangle = 0,
	\]
where $\langle\cdot,\cdot\rangle$ denotes the dual pairing between distributions and smooth functions on $\mathbb{T}^3$.
\end{itemize}
\end{definition} 
\begin{remark}
In particular, we emphasize that it is required $\hat{B}(0) = B_{*}$, 
which ensures that the constructed $B$ can not be trivial when $B_{*}\neq 0$.
\end{remark}
\begin{remark}
Here, we use $H^{-s}$ instead of $\dot{H}^{-s}$ for the tensor products 
in order to avoid a $\div$ evaluated at infinity, which is different from the ones in \cite{ABGM25} and \cite{CH26}. 
\end{remark}

Next, we present Theorem \ref{thm2},
which demonstrates the existence of infinitely many singular solutions for the stationary MHD equations.
\begin{theorem}[Existence of non-trivial stationary singular solutions]
    \label{thm2}
	Given $$0<\theta<1,q,r \in [1,\infty],$$
	there exist infinitely many weak solutions
	$$
	(u,B)\in
	B^{-\theta}_{q,r} (\mathbb{T}^3)
	$$
	 to the 3-D stationary fractional MHD equations \eqref{equa-sMHD}. 
\end{theorem}
\begin{remark}
Following the proof of  Sections 5 in \cite{CH26}, for $\alpha_1,\alpha_2 > 1/2$ 
we can also get two different mild solutions to the non-stationary MHD equations. 
\end{remark}

\subsection{Organization.}
Section 2 presents the iteration theorem.
Sections 3 and 4 construct perturbations using the convex integration method and prove the iteration theorem. 
Section 5 demonstrates the proof of the main theorem.
\subsection{Notation.}
In this part, we provide an explanation for some of the main notations used throughout this paper:
\begin{enumerate}[label=\normalfont(\roman*)]
	\item  $\mathbb{P}$ denotes the Helmholtz-Leray projector, i.e., $\mathbb{P}=\Id-\nabla\Delta^{-1}\div$.
	\item  $X\lesssim Y$ means that $X\leq C Y$ for some constant $C>0$.
	\item  $u\otimes v:=(u_iv_j)_{1\leq i,j\leq 3},$ for smooth $u,v$ and is defined as in \eqref{e:u*v} for singular ones.
	\item  ${P}_{> 0}f := f- \aint_{\mathbb{T}^3}f dx $ denotes projection onto nonzero Fourier modes.
\end{enumerate}

\section{Iterative procedure}
{\noindent\bf $\bullet$ Main Iteration.}
We consider the approximate solutions to the following incompressible stationary MHD-Reynolds system
\begin{equation}\label{R_MHDs}
	\left\{\aligned
    & \div ( u_{q} \otimes u_{q}  - B_{q} \otimes B_{q} ) + \nu_1 (-\Delta)^{\alpha_1 } u_{q} +\nabla p_{q} = \div \mathring{R}_{q}^u \\
     &\div ( B_{q} \otimes u_{q}  - u_{q} \otimes B_{q} )  + \nu_2 (-\Delta)^{\alpha_2 } B_{q} = \div \mathring{R}_{q}^B\\
     &\div u_{q} =  \div B_{q} = 0,\\
     &{\hat{u}}_q (0) = 0, \quad {\hat{B}}_q (0) = B_{*}.
	\endaligned
    \right.
\end{equation}
where $q \in \mathbb{N}$, $ \mathring{R}_{q}^u$ is the Reynolds stress which is a trace-free symmetric $3\times3$ matrix,
and $ \mathring{R}_{q}^B$ is the magnetic Reynolds stress which is a skew-symmetric $3\times3$ matrix.

In the construction of the perturbation, the frequency parameter $\lambda_q$ and the amplitude parameter $\delta_{q}$ are chosen in the following way
\begin{equation}\label{la}
	\la=a^{(b^q)}, \ \
	\delta_{q}= \lbb_1^{3\gamma} \lambda_{q}^{-2\gamma}.
\end{equation}
For given $\theta$, we choose $\delta \ll \frac{\theta}{3} $, then we choose $\varepsilon$ and $b$ satisfies
\begin{align} \label{epsilon_condition}
\varepsilon \ll \delta \ll \theta \ll 1, \quad b\varepsilon \in \mathbb{N}, 
\end{align}
where $b$ is a large integer.  
Moreover we will select appropriate parameters $\gamma$ that is a very small regularity parameter 
such that it satisfies $\gamma b^2 \ll 2 b\gamma \le \varepsilon$.
Finally we determine 
$a\in \mathbb{N}$ that is a large dyadic integer. 

For smooth approximate solutions $(u_q,B_q,\mathring{R}^{u}_q,\mathring{R}^{B}_q)$, we assume the following iterative estimates:
\begin{align}
	& \|(u_q, B_q)\|_{L^\infty} \lesssim \lambda_{q}^{2}, \label{mitr1_} \\
	& \|(u_q, B_q)\|_{C^N} \lesssim \lambda_q^{6N}, \quad 1 \le N \le 7, \label{mitr2_} \\
    & \|(u_q, B_q)\|_{B^0_{\infty,1}} \lesssim \lambda_q^{6},  \label{mitr6_} \\
	& \|(\mathring{R}^{u}_q, \mathring{R}^{B}_q)\|_{L^\infty}\lesssim  \lambda_{q}^{12\alpha},\label{mitr3_}	\\
    & \|(\mathring{R}^{u}_q, \mathring{R}^{B}_q)\|_{C^N} \lesssim \lambda_q^{12N\alpha}, \quad 1 \le N \le 7,  \label{mitr4_}\\
	& \|(\mathring{R}^{u}_q, \mathring{R}^{B}_q)\|_{\Hs} \leq  \delta_{q+1}, \label{mitr5_}
\end{align}
where $\alpha = \max\{\alpha_1,\alpha_2 ,1\}$.
\begin{theorem}[Main iteration]
    \label{iteration}
	Let $0 < \theta < 1$, $s > 2 \alpha +3/2$ and  $N_q$ be a dyadic integer.
    Given a smooth solution $(u_q,B_q,\mathring{R}^{u}_q,\mathring{R}^{B}_q)$ to \eqref{R_MHDs} with $\supp {\hat{u}}_q \subset B(0,N_q)$ and 
	 $\supp {\hat{B}}_q \subset B(0,N_q),$ 
	there exists another solution $(u_{q+1},B_{q+1},\mathring{R}^{u}_{q+1},\mathring{R}^{B}_{q+1})$ such that 
    \begin{align}
	& \|(u_{q+1}, B_{q+1})\|_{L^\infty} \lesssim \lambda_{q+1}^{2}, \label{mitr1} \\
	& \|(u_{q+1}, B_{q+1})\|_{C^N} \lesssim \lambda_{q+1}^{6N}, \quad 1 \le N \le 7, \label{mitr2} \\
    & \|(u_{q+1}, B_{q+1})\|_{B^0_{\infty,1}} \lesssim \lambda_{q+1}^{6},  \label{mitr6} \\
	& \|(\mathring{R}^u_{q+1}, \mathring{R}^B_{q+1})\|_{L^{\infty}}\lesssim  \lambda_{q+1}^{12 \alpha },\label{mitr3}	\\
    & \|(\mathring{R}^{u}_{q+1}, \mathring{R}^{B}_{q+1})\|_{C^N} \lesssim \lambda_{q+1}^{12N\alpha}, \quad 1 \le N \le 7,  \label{mitr4}\\
	& \|(\mathring{R}^u_{q+1}, \mathring{R}^B_{q+1})\|_{\Hs} \leq  \delta_{q+2}. \label{mitr5}
\end{align}
    Moreover, the velocity perturbation $w_{q+1}:=u_{q+1}-u_q$  and magnetic perturbation $d_{q+1}:=B_{q+1}-B_q$ satisfy
    \begin{enumerate}[label=\normalfont(\roman*)]
        \item \label{item:prop-hatw-supp}
         There exists a dyadic number $N_{q+1}$ with $N_{q+1} \gg N_{q}$ so that
        \begin{align}
        \supp \widehat{w}_{q+1} \subset \left\{ k \in \mathbb{Z}^3: \frac{1}{2} N_{q+1} < |k| < \frac{9}{10}  N_{q+1} \right\}, \\
        \supp \widehat{d}_{q+1} \subset \left\{ k \in \mathbb{Z}^3: \frac{1}{2} N_{q+1} < |k| < \frac{9}{10}  N_{q+1} \right\}
        \end{align} 
        and
        \begin{align}
            \supp {\hat{u}}_{q+1} \subset B(0,N_{q+1}),\\
            \supp {\hat{B}}_{q+1} \subset B(0,N_{q+1});
        \end{align}

        \item \label{item:prop-w-Lp}
        The $B^{-\theta}_{\infty,1}$-estimate holds
        \begin{align}
            \|w_{q+1}\|_{B^{-\theta}_{\infty,1}} + \|d_{q+1}\|_{B^{-\theta}_{\infty,1}} < \delta_{q+2}; 
         \end{align} 

        \item \label{item:prop-w-paraproduct}
        The paraproducts have estimates  
        \begin{align}
\|w_{q+1} \otimes w_{q+1}\|_{H^{-s}} &+ \sum_{M \le 2N_q} \Big(\|P_M u_q \otimes w_{q+1}\|_{H^{-s}}  + \|w_{q+1} \otimes P_M u_q\|_{H^{-s}} \Big)\notag\\
 &\lesssim \delta_{q+1},\\
\|w_{q+1} \otimes d_{q+1}\|_{H^{-s}} &+ \sum_{M \le 2N_q} \Big(\|P_M u_q \otimes d_{q+1}\|_{H^{-s}} + \|w_{q+1} \otimes P_M B_q\|_{H^{-s}} \Big) \notag\\
+\|d_{q+1} \otimes w_{q+1}\|_{H^{-s}} &+ \sum_{M \le 2N_q} \Big(\|P_M B_q \otimes w_{q+1}\|_{H^{-s}} + \|d_{q+1} \otimes P_M u_q\|_{H^{-s}} \Big) \notag\\
 &\lesssim  \delta_{q+1}, \\
\|d_{q+1} \otimes d_{q+1}\|_{H^{-s}} &+ \sum_{M \le 2N_q} \Big(\|P_M B_q \otimes d_{q+1}\|_{H^{-s}}+ \|d_{q+1} \otimes P_M B_q\|_{H^{-s}} \Big) \notag\\
 &\lesssim  \delta_{q+1}.
\end{align}

    \end{enumerate}
\end{theorem}

\section{Perturbations}
\subsection{Geometric Lemma}
The following two lemmas are an important part of designing the perturbations. 
\begin{lemma}[\cite{bbv20}\textit{Lemma 4.1}]
\label{geometric lem 1}
There exists a finite set $\Lambda_B \subset S^2 \cap \mathbb{Q}^3$ that consists of vectors $\xi$ with associated orthonormal bases $(\xi, \xi_1, \xi_2) \subset S^2 \cap \mathbb{Q}^3$,  $\varepsilon_B > 0$, 
and smooth positive functions $\gamma_{(\xi)}: B_{skew3}(0,\varepsilon_B) \to \mathbb{R}$, where $B_{skew3}(0,\varepsilon_B)$ is the ball of radius $\varepsilon_B$ centered at 0 in the space of $3 \times 3$ skew-symmetric matrices, such that for  $A \in B_{skew3}(0,\varepsilon_B)$ we have the following identity: 
\begin{equation}
\label{antisym}
A = \sum_{\xi \in \Lambda_B} \gamma_{(\xi)}^2(A) (\xi_1 \otimes \xi_2 - \xi_2 \otimes \xi_1)  \,.
\end{equation}
\end{lemma}
\begin{lemma}[\cite{bbv20}\textit{Lemma 4.2}]
\label{geometric lem 2}
There exists a finite set $\Lambda_u \subset (S^2 \cap \mathbb{Q}^3)\setminus\Lambda_B$ that consists of vectors $\xi$ with associated orthonormal bases $(\xi, \xi_1, \xi_2) \subset S^2 \cap \mathbb{Q}^3$,  $\varepsilon_u> 0$, 
and smooth positive functions $\gamma_{(\xi)}: B_{sym3}(\Id,\varepsilon_u) \to \mathbb{R}$, where $B_{sym3}(\Id,\varepsilon_u)$ is the ball of radius $\varepsilon_u$ centered at the identity in the space of $3 \times 3$ symmetric matrices,  such that for  $S \in B_{sym3}(\Id,\varepsilon_u)$ we have the following identity:
\begin{equation}
\label{sym}
S = \sum_{\xi \in \Lambda_u} \gamma_{(\xi)}^2(S) \xi_1 \otimes \xi_1   \,.
\end{equation} 
\end{lemma}
\begin{remark}
\label{Universal constant}
By any given choice of $\Lambda_B$ and $\Lambda_u$ and the associated orthonormal bases, there exists $N_{\Lambda} \in \mathbb{N}$ with
\begin{equation*}
\{ N_{\Lambda} \xi,N_{\Lambda}\xi_1 , N_{\Lambda}\xi_2 \ | \ \xi \in \Lambda_u \cup \Lambda_B\} \subset N_{\Lambda} \mathbb{S}^2 \cap \mathbb{Z}^3.
 \end{equation*} 
\end{remark}


\begin{remark}
For a large integer $N_{*} \gg 1$, there exists a geometric constant  $M_* > 0$ such that
\begin{align}
\sum_{\xi \in \Lambda_{u}} \|\gamma_{(\xi)}\|_{C^{N_{*}}(B_{\mathrm{sym}3}(\mathrm{Id},\varepsilon_u))} + \sum_{\xi \in \Lambda_{B}} \|\gamma_{(\xi)}\|_{C^{N_{*}}(B_{\mathrm{skew}3}(0,\varepsilon_B))} \leq M_* \,.
\label{M bound}
\end{align}
\end{remark}

\subsection{Intermittent flows}
The intermittent flows are composed of spatial building blocks, 
which are characterized by four parameters $\rs$, $\lambda$, $\tau$ and $\sigma$:
\begin{equation}\label{choose-prar}
	\rs := \lambda_{q+1}^{-10\varepsilon},\
	\lambda := \lambda_{q+1},\   \tau:=\lambda_{q+1}^{2}, \ \sigma:=\lambda_{q+1}^{5}>50\tau^{2}.
\end{equation}

Next, we mainly recall from \cite{bbv20} their constructions and properties.

Let $\phi : \mathbb{R} \to \mathbb{R}$ be a smooth cut-off function supported on
the interval $[-1,1]$
and normalize $\phi$ to satisfy
\begin{equation}\label{e4.91}
	\frac{1}{2 \pi}\int_{\mathbb{R}} \phi^2(x)\d x = 1,
\end{equation}
and 
\begin{equation}\label{e4.92}
   \frac{1}{2 \pi}\int_{\mathbb{R}} \phi(x) \d x = 0.
\end{equation}
The corresponding rescaled cut-off functions are defined by
\begin{equation*}
	\phi_{\rs}(x) := {\rs^{-\frac{1}{2}}}\phi\left(\frac{x}{\rs}\right), 
\end{equation*}
Here,  $\phi_{\rs}$ is supported on the interval $[-r,r]$.
By an abuse of notation,
we periodize $\phi_{\rs}$ so that
they are treated as periodic functions defined on $\mathbb{T}$.

The \textit{intermittent shear flows} are defined by
\begin{equation*}
	W_{(\xi)} := \phi_{\rs}( \lambda \rs N_{\Lambda}\xi\cdot x)\xi_1,\ \  \xi \in \Lambda_u \cup \Lambda_{B}  ,
\end{equation*}
and the \textit{intermittent magnetic shear flieds} are defined by
\begin{equation*}
	D_{(\xi)} := \phi_{\rs}( \lambda \rs N_{\Lambda}\xi\cdot x)\xi_2, \ \ \xi \in \Lambda_{B}  .
\end{equation*}
Here, $N_{\Lambda}$ is given in Remark \ref{Universal constant},
$(\xi,\xi_1,\xi_2)$ are the orthonormal bases in $\R^3$ in
geometric lemmas, $\Lambda_u, \Lambda_{B}$ are the wave vector sets of finite cardinality,
and the parameter $\rs$  parameterizes the concentration of the flows.
In particular,
$\{W_{(\xi)}, D_{(\xi)}\}$ are $(\mathbb{T}/(\lbb \rs))^3$-periodic, supported on thin planes with thickness $\sim {1}/{\lbb}$ in each periodic domain.

To simplify notations, we set
\begin{align}\label{snp-endpt2}
	\phi_{(\xi)}(x) := \phi_{\rs}(\lambda \rs N_{\Lambda}\xi\cdot x),  \ \
\end{align}
and rewrite
\begin{subequations}
	\begin{align}
		&	W_{(\xi)} = \phi_{(\xi)} \xi_1,\quad  \xi\in \Lambda_u\cup \Lambda_{B}, \label{snwd-endpt2}\\
		&	D_{(\xi)} = \phi_{(\xi)} \xi_2,\quad  \xi\in \Lambda_{B}.  \label{snd-endpt2}
	\end{align}
\end{subequations}
Then,
$W_{(\xi)}$ and $D_{(\xi)}$ are  mean zero on $\T^3$. Moreover, the corresponding tenser potentials are defined by
\begin{subequations} \label{corrector vector-endpt2}
	\begin{align}
		& \widetilde{W}_{(\xi)}^{c} := \xi_1 \otimes \nabla \Delta^{-1} \phi_{(\xi)} - \nabla \Delta^{-1} \phi_{(\xi)} \otimes \xi_1,\\
		& \widetilde{D}_{(\xi)}^{c} := \xi_2 \otimes \nabla \Delta^{-1} \phi_{(\xi)} - \nabla \Delta^{-1} \phi_{(\xi)} \otimes \xi_2.
	\end{align}
\end{subequations}
where $\Delta^{-1}$ is defined by the Fourier transform for mean-free functions:
\[
\widehat{\Delta^{-1} f}(k) = -\frac{1}{|k|^2} \hat{f}(k), \quad k \neq 0.
\]
Using $(\xi_i \cdot \nabla) \phi_{(\xi)} = 0,i=1,2$ and $\fint \phi_{(\xi)} = 0$, we get
\[ \div \widetilde{W}_{(\xi)}^c = \div(\nabla \Delta^{-1} \phi_{\xi}) \xi_1 - (\xi_1 \cdot \nabla) \nabla \Delta^{-1} \phi_{(\xi)} = \phi_{(\xi)} \xi_1 - 0 = W_{(\xi)}, \]
\[ \div \widetilde{D}_{(\xi)}^c = \div(\nabla \Delta^{-1} \phi_{\xi}) \xi_2 - (\xi_2 \cdot \nabla) \nabla \Delta^{-1} \phi_{(\xi)} = \phi_{(\xi)} \xi_2 - 0 = D_{(\xi)}. \]

Here, the intermittent flows $W_{(\xi)}$ and $D_{(\xi)} $ will be key ingredients 
in building the principal perturbations $ w^{(p)}_{q+1} $ and $  d^{(p)}_{q+1} $.
Moreover, their associated potentials will be employed to 
build the incompressible correctors $  w^{(c)}_{q+1} $  and $  d^{(c)}_{q+1} $. Lemma \ref{buildingblockestlemma} below contains their important analytic estimates.

\begin{lemma} [\cite{bbv20}\textit{Lemmas 5.1-5.2}] 
    \label{buildingblockestlemma}
	For any $p \in [1,\infty]$ and $N \in \mathbb{N}$, we have
	\begin{align}
		&\left\|\nabla^{N} \phi_{(\xi)}\right\|_{L^{p}}
		\lesssim r_{\perp}^{\frac 1p- \frac12}  \lambda^{N}.  \label{intermittent estimates2-endpt2}
	\end{align}
	For any $p \in (1,\infty)$ and $N \in \mathbb{N}$, we have
	\begin{align}
		&\displaystyle \|\nabla^{N}  W_{(\xi)}\|_{L^{p}}
		+\lambda \|\nabla^{N}  \tilde{W}_{(\xi)}^c\|_{L^{p}}\lesssim r_{\perp}^{\frac 1p- \frac12} \lambda^{N},
         \ \ \xi\in \Lambda_u \cup \Lambda_{B}, \label{ew-endpt2} \\
        &\displaystyle \|\nabla^{N}  D_{(\xi)}\|_{ L^{p}}
         +\lambda \|\nabla^{N} \tilde{D}_{(\xi)}^c\|_{L^{p}}\lesssim r_{\perp}^{\frac 1p- \frac12} \lambda^{N},
         \ \ \xi\in \Lambda_{B}. \label{ed-endpt2}
	\end{align}
	Moreover, for every $\xi \neq \xi'\in \Lambda_{u}\cup \Lambda_{B}$, $N\in \mathbb{N}$ and $p \in [1, \infty]$,
	the following product estimate holds
	\begin{equation}  \label{intersect-phik1}
		\|\na^N (\phi_{(\xi)}\phi_{(\xi')}) \|_{L^p}\lesssim \lbb^N \rs^{\frac{2}{p}-1}.
	\end{equation}
\end{lemma}

\subsection{Amplitudes of perturbations $a_{(\xi)}$}   \label{Subsec-Amplitude-S1}
The amplitudes of the perturbations are constructed mainly to decrease the effects of $\mathring{R}^u_{q}$ and $\mathring{R}^B_{q}$ 
based on geometric lemmas.\\
{\bf $\bullet$ Amplitudes of magnetic perturbations.}
Let $\chi: [0, +\infty) \to \mathbb{R}$ be a smooth cut-off function such that
\begin{equation}\label{e4.0}
	\chi (z) =
	\left\{\aligned
	& 1,\quad 0 \leq z\leq 1, \\
	& z,\quad z \geq 2,
	\endaligned
	\right.
\end{equation}
and
\begin{equation}\label{e4.1}
	\frac 12 z \leq \chi(z) \leq 2z \quad \text{for}\quad z \in (1,2).
\end{equation}
Set
\begin{equation}\label{rhob}
	\varrho_B(x) := 2 \varepsilon_B^{-1}  \lbb_q^{-\frac{\ve_R}{4}}
      \delta_{q+ 1} \chi\left( \frac{|\mathring{R}_{q}^{B}(x) |}{\lbb_q^{-\frac{\ve_R}{4}}  \delta_{q+1} } \right),
\end{equation}
where $\varepsilon_{B}$ is the small radius in Lemma \ref{geometric lem 1},
and $\mathring{R}_{q}^{B}(x)$ is the magnetic stress.
Then, one has
\begin{equation}\label{rhor}
	\left|  \frac{\mathring{R}_{q}^{B}}{\varrho_B} \right|
	= \left| \frac{\mathring{R}_{q}^{B}}{2 \varepsilon_B^{-1} \lbb_q^{-\frac{\ve_R}{4}}  \delta_{q+ 1}\chi
		( \lbb_q^{\frac{\ve_R}{4}} \delta_{q+1} ^{-1} |\mathring{R}_{q}^{B} | )} \right| \leq \varepsilon_B,
\end{equation}

Then, the amplitudes of the magnetic perturbations are defined by
\begin{equation}\label{akb}
	a_{{(\xi)}}(x):= a_{\xi, B}(x)
	=  \sqrt{2}\varrho_B^{\frac{1}{2} } (x) \gamma_{{(\xi)}}
	\left(\frac{-\mathring{R}_{q}^{B}(x)}{\varrho_B(x)}\right), \quad \xi \in \Lambda_B,
\end{equation}
where $\gamma_{(\xi)}$ is the smooth function in  Lemma~\ref{geometric lem 1}.

\begin{lemma} \label{Lem-mae-S1}
	For $1\leq N\leq 7$, $\xi\in \Lambda_B$, denote $\ell := \lambda_q^{-30}$, we have 
	\begin{align}    \label{a-mag-S1}
		\norm{a_{(\xi)}^2}_{H^{-s}} \lesssim \delta_{q+1} ,\ \
	 \norm{ a_{(\xi)} }_{C^N} \lesssim \ell^{-N\alpha}.
	\end{align}
     Moreover, we have 
    \begin{align}\label{a-vel-S1-1}
		\norm{a_{(\xi)}}_{L^{\infty}} \lesssim \ell^{-\alpha}.
	\end{align}
\end{lemma}
See \cite{bbv20}, we can find more detials of the proof.
\begin{proof}
By definition, we can have 
\begin{align*}
     \|a_{(\xi)}^2\|_{\Hs}  = \sup_{\|f\|_{H^s} \le 1} \frac{\langle a_{(\xi)}^2 , f \rangle }{ \|f\|_{H^s} } &\lesssim   \sup_{\|f\|_{H^s} \le 1} \big(\frac{\langle 2 \varepsilon_B^{-1} \lbb_q^{-\frac{\ve_R}{4}}  \delta_{q+ 1} , f \rangle }{ \|f\|_{H^s} } + \frac{\langle 2 \varepsilon_B^{-1} \mathring{R}_{q}^{B} , f \rangle }{ \|f\|_{H^s} })\\
    & \lesssim  \lbb_q^{-\frac{\ve_R}{4}}  \delta_{q+ 1} + \| \mathring{R}_{q}^{B} \|_{\Hs},
\end{align*}
So we can have 
\begin{align*}
    \|a_{(\xi)}^2\|_{H^{-s}}  \lesssim \delta_{q + 1}.
\end{align*}    
    By using standard H\"older estimates (see \cite{bdis15} and \cite{lqzz22} for more details), \eqref{mitr3_}, \eqref{mitr4_}, we can calculate the $\norm{ a_{(\xi)} }_{C^{N}}$ as 
\begin{align*}
    \norm{ a_{(\xi)} }_{C^{N}} 
    \lesssim \ell^{-N\alpha},\quad 1\le N \le 7.
\end{align*} 
Moreover, 
\begin{equation*}
     \norm{ a_{(\xi)} }_{L^{\infty}} \lesssim  \norm{ a_{(\xi)} }_{C^{1}} \lesssim \ell^{-\alpha}. \qedhere
\end{equation*} 
\end{proof}

{\noindent\bf $\bullet$ Amplitudes of velocity perturbations.}
First, denote 
\begin{equation}
	\label{def:G}
	\mathring{G}^{B}: = \sum_{\xi \in \Lambda_B}a_{(\xi)}^2 \aint_{\mathbb{T}^3} W_{(\xi)} \otimes W_{(\xi)} - D_{(\xi)} \otimes D_{(\xi)}\d x.
\end{equation}
By \eqref{a-mag-S1}, it holds that, for $1 \le N \le 7$,
\begin{align}
	\label{G estimates}
    \norm{\mathring{G}^{B}}_{H^{-s}} \lesssim \delta_{q+1},\quad
	\norm{\mathring{G}^{B}}_{C^0} \lesssim \ell^{-2\alpha},\quad
	\norm{\mathring{G}^{B}}_{C^N} \lesssim \ell^{-2N\alpha} ,
\end{align}
Set
\begin{equation}\label{defrho}
	\begin{aligned}
		&  \varrho_u(x):= 2 \varepsilon_u^{-1}  \delta_{q+1}\chi
		\left( \frac{|\mathring{R}_{q}^{u}(x) + \mathring{G}^{B}(x)|}{\delta_{q+1}} \right).
	\end{aligned}
\end{equation}
By \eqref{e4.0} and \eqref{e4.1}, it holds that
\begin{align}  \label{R-G-veu}
	& \left|  \frac{\mathring{R}_{q}^{u}+ \mathring{G}^{B}}{\varrho_u} \right| \leq \varepsilon_u,
\end{align}
where $\varepsilon_{u}$ is the small radius in Lemma \ref{geometric lem 2},
and $\mathring{R}_{q}^{u}(x)$ is the Reynolds stress.

The amplitudes of the velocity perturbations can be defined by
\begin{equation}\label{velamp}
	\begin{aligned}
		&a_{(\xi)} (x) := a_{\xi, u}(x)
		= \sqrt{2} \varrho_u^{\frac{1}{2}} \gamma_{(\xi)}
		\left(\Id - \frac{\mathring{R}_{q}^{u}(x) +  \mathring{G}^{B}(x)}{\varrho_u(x)} \right),
		\quad \xi \in \Lambda_u.
	\end{aligned}
\end{equation}

We show the analytic properties of
the velocity amplitudes $\{a_{(\xi)}:\xi\in \Lambda_u\}$
in the following lemma.
\begin{lemma}\label{Lem-vae-S1}
	For $1\leq N\leq 7$, $\xi\in \Lambda_{u} $, 
	we have
	\begin{align}\label{a-vel-S1-2}
		\norm{a_{(\xi)}^2}_{H^{-s}} \lesssim \delta_{q+1} ,\ \
		\norm{a_{(\xi)}  }_{C^N} \lesssim \ell^{-2N\alpha}.
	\end{align}
    Moreover, we have 
    \begin{align}\label{a-vel-S1-3}
		\norm{a_{(\xi)}  }_{L^{\infty}} \lesssim \ell^{-2\alpha}.
	\end{align}
\end{lemma}
The proof of Lemma \ref{Lem-vae-S1} is similar to  Lemma \ref{Lem-mae-S1}.
\begin{proof}
    Using \eqref{G estimates} and the same calculations as Lemma \ref{Lem-mae-S1}, we can have
    $$\norm{a_{(\xi)}^2}_{H^{-s}} \lesssim \delta_{q+1}.$$
    Moreover, we can get the estimation to $\norm {a_{(\xi)}}_{C^N}$ directly by using \eqref{mitr3_} -- \eqref{mitr4_} and \eqref{G estimates},
\begin{equation*}
    \norm{ a_{(\xi)} }_{C^N} \lesssim \ell^{-2N\alpha},\quad
    \norm{ a_{(\xi)} }_{L^{\infty}} \lesssim \norm{ a_{(\xi)} }_{C^1} \lesssim \ell^{-2\alpha}.\qedhere
\end{equation*}
\end{proof}

\subsection{Velocity and magnetic perturbations.}
Now, we can construct velocity and magnetic perturbations.\\
{\bf $\bullet$ Principal parts.} Firstly, we define the principal parts $w_{q+1}^{(p)}$ and $d_{q+1}^{(p)}$:
\[ w^{(p)}_{q+1}(x) := \sum_{\xi \in \Lambda_{u}\cup \Lambda_{B}} a_{(\xi)}(x) W_{(\xi)}(x) \cos(2\pi\sigma_\xi \xi \cdot x), \]
\[ d^{(p)}_{q+1}(x) := \sum_{\xi \in \Lambda_{B}} a_{(\xi)}(x) D_{(\xi)}(x) \cos(2\pi\sigma_\xi \xi \cdot x). \]
These perturbations are used to generate low-frequency components to eliminate the effects of the last Reynolds and magnetic stress tensors. 
Here we introduce high oscillation part $\cos(2\pi\sigma_\xi \xi \cdot x)$ which ensures that the support of the Fourier transform to the low-frequency error lies on a high-frequency shells. 
Here, $\sigma_{\xi}$ are oscillation parameters. To show the cancellation, we compute the following identity $\mathring{R}_{{\mathrm{osc}}}^B$ as
\begin{align}
	\mathring{R}^B_{{\mathrm{osc}}} &= d_{q+ 1}^{(p)} \otimes w_{q+ 1}^{(p)} - w_{q+1}^{(p)}\otimes d_{q+ 1}^{(p)}\notag + \rb   \nonumber \\
	&= \sum_{\xi \in \Lambda_B} \frac{1}{2} a_{(\xi)}^2  \P_{> 0}(D_{(\xi)}\otimes W_{(\xi)}-W_{(\xi)}\otimes D_{(\xi)}) \nonumber \\
	&+ \sum_{\xi \in \Lambda_B} \frac{1}{2} \cos(4\pi \sigma_{\xi}{\xi} \cdot x)  a_{(\xi)}^2 (D_{(\xi)}\otimes W_{(\xi)}-W_{(\xi)}\otimes D_{(\xi)}) \nonumber \\
	&+ \(\sum_{\xi \neq \xi' \in \Lambda_B}+ \sum_{\xi \in \Lambda_u,\xi' \in \Lambda_B}\)a_{(\xi)}a_{(\xi')}\cos(2\pi \sigma_{\xi}{\xi} \cdot x) \cos(2\pi \sigma_{\xi'}{\xi^{'}} \cdot x)  \nonumber \\
    &(D_{(\xi')}\otimes W_{(\xi)}-W_{(\xi)}\otimes D_{(\xi^{'})}),
\end{align}
and $\mathring{R}_{{\mathrm{osc}}}^u$ as
\begin{align}  
	\mathring{R}_{{\mathrm{osc}}}^u & = w_{q+ 1}^{(p)} \otimes  w_{q+ 1}^{(p)} - d_{q+ 1}^{(p)} \otimes d_{q+1}^{(p)} + \ru \nonumber  \\
	=&  \varrho_u \Id    +\sum_{\xi \in \Lambda_u} \frac{1}{2} a_{(\xi)}^2 \P_{> 0}( W_{(\xi)}\rotimes W_{(\xi)}) \nonumber \\
	& + \sum_{\xi \in \Lambda_B} \frac{1}{2} a_{(\xi)}^2 \P_{> 0} (W_{(\xi)}\rotimes W_{(\xi)}-D_{(\xi)}\rotimes D_{(\xi)})\nonumber \\
	& +\sum_{\xi \in \Lambda_u} \frac{1}{2}  \cos(4\pi \sigma_{\xi}{\xi} \cdot x) a_{(\xi)}^2 ( W_{(\xi)}\rotimes W_{(\xi)}) \nonumber \\
	& + \sum_{\xi \in \Lambda_B} \frac{1}{2}  \cos(4\pi \sigma_{\xi}{\xi} \cdot x) a_{(\xi)}^2 (W_{(\xi)}\rotimes W_{(\xi)}-D_{(\xi)}\rotimes D_{(\xi)})\nonumber \\
	& + \sum_{\xi \neq \xi' \in \Lambda_u\cup\Lambda_B }a_{(\xi)}a_{(\xi')} \cos(2\pi \sigma_{\xi}{\xi} \cdot x) \cos(2\pi \sigma_{\xi'}{\xi^{'}} \cdot x) W_{(\xi)}\rotimes W_{(\xi')}\nonumber \\
	& -\sum_{\xi \neq \xi' \in \Lambda_B}a_{(\xi)}a_{(\xi')}\cos(2\pi \sigma_{\xi}{\xi} \cdot x) \cos(2\pi \sigma_{\xi'}{\xi^{'}} \cdot x)D_{(\xi)}\rotimes D_{(\xi')}.
\end{align}
{\bf $\bullet$ Localization corrector.}
In order to make sure that $\supp \widehat{w}_{q+1}$ and $\supp \widehat{d}_{q+1}$ lie in certain shells in the frequency space, we need to 
introduce the localization corrector $w^{(l)}_{q+1}$ and $d^{(l)}_{q+1}$ as
\begin{align*}
    w^{(l)}_{q+1}(x) := &-\sum_{\xi \in \Lambda_{u}\cup \Lambda_{B}} a_{(\xi)}(x) P_{>\tau} (W_{(\xi)})(x) \cos(2\pi\sigma_\xi \xi \cdot x) \\
    &-\sum_{\xi \in \Lambda_{u}\cup \Lambda_{B}} P_{>\tau} (a_{(\xi)})(x) P_{\le \tau} (W_{(\xi)})(x)\cos(2\pi\sigma_\xi \xi\cdot x),\\
    d^{(l)}_{q+1}(x) := &-\sum_{\xi \in \Lambda_{B}} a_{(\xi)}(x) P_{>\tau} (D_{(\xi)})(x) \cos(2\pi\sigma_\xi \xi\cdot x) \\
    &-\sum_{\xi \in \Lambda_{B}} P_{>\tau} (a_{(\xi)})(x) P_{\le \tau} (D_{(\xi)})(x)\cos(2\pi\sigma_\xi \xi \cdot x).
\end{align*}
Thus,
    \[w^{(p)}_{q+1} + w^{(l)}_{q+1} = \sum_{\xi \in \Lambda_{u}\cup \Lambda_{B}} P_{\le \tau} (a_{(\xi)})(x) P_{\le \tau} (W_{(\xi)})(x) \cos(2\pi \sigma_\xi \xi \cdot x),\]

    \[d^{(p)}_{q+1} + d^{(l)}_{q+1} = \sum_{\xi \in \Lambda_{B}} P_{\le \tau} (a_{(\xi)})(x) P_{\le \tau} (D_{(\xi)})(x) \cos(2\pi \sigma_\xi \xi \cdot x).\]
\begin{lemma}[\cite{CH26}\textit{Lemma 6.2}]
    \label{lemma:arithmetic-sigma}
    For any dyadic integer $\tau \ge 1$, there exist integers $\{\sigma_\xi\}_{\xi \in \Lambda}$ and a dyadic number $\sigma>50\tau^{3}$ so that
    \[ \bigcup_{\xi \in \Lambda} \big(B(\sigma_\xi \xi , 2\tau) \cup B(-\sigma_\xi \xi , 2\tau) \big) \subset \left\{k \in \mathbb{Z}^3:\frac{1}{2} \sigma < |k| < \frac{9}{10} \sigma \right\}. \]
\end{lemma}
Using this lemma, we can have 
\[
        \supp (\widehat{w}^{(p)}_{q+1} + \widehat{w}^{(l)}_{q+1} ) \subset \left\{  k \in \mathbb{Z}^3:\frac{1}{2} \sigma < |k| < \frac{9}{10} \sigma \right\}.
\]
\[
        \supp (\widehat{d}^{(p)}_{q+1} + \widehat{d}^{(l)}_{q+1}) \subset \left\{ k \in \mathbb{Z}^3:\frac{1}{2} \sigma < |k| < \frac{9}{10} \sigma \right\}.
\]
{\noindent\bf $\bullet$ Incompressibility correctors.}
To make sure $w_{q+1}$ and $d_{q+1}$ are divergence-free, we need to introduce the divergence-free corrector $w^{(c)}_{q+1}$ and $d^{(c)}_{q+1}$ . Define
\[ \Phi_{(\xi)} := P_{\le \tau} (\phi_{(\xi)})(x) \cos(2\pi \sigma_\xi \xi \cdot x), \]
and
\[ {\tilde{W}_{(\xi)}}  := \Phi_{(\xi)} \xi_1 = P_{\le \tau} ( W_{(\xi)} ) (x)\cos(2\pi \sigma_\xi \xi \cdot x),\]
\[ {\tilde{D}_{(\xi)}}  := \Phi_{(\xi)} \xi_2 = P_{\le \tau} ( D_{(\xi)} ) (x)\cos(2\pi \sigma_\xi \xi \cdot x). \]
Define the anti-symmetric tensors ${W^{(c)}_{(\xi)}} $ and ${D^{(c)}_{(\xi)}} $  by
\[ {W^{(c)}_{(\xi)}}:= \xi_1 \otimes \nabla \Delta^{-1} \Phi_{(\xi)} - \nabla \Delta^{-1} \Phi_{(\xi)} \otimes \xi_1, \]
\[ {D^{(c)}_{(\xi)}}:= \xi_2 \otimes \nabla \Delta^{-1} \Phi_{(\xi)} - \nabla \Delta^{-1} \Phi_{(\xi)} \otimes \xi_2, \]
and the divergence-free corrector $w^{(c)}_{q+1}$ and $d^{(c)}_{q+1}$by
\[ w^{(c)}_{q+1} := \sum_{\xi \in \Lambda_u\cup\Lambda_B}{W^{(c)}_{(\xi)}} \  \nabla P_{\le \tau} (a_{(\xi)}) , \]
\[ d^{(c)}_{q+1} := \sum_{\xi \in \Lambda_B} {D^{(c)}_{(\xi)}} \ \nabla P_{\le \tau} (a_{(\xi)}).\]

\begin{lemma}
    \label{lemma:w}
    Given $\tau,\sigma \ge1$ as in Lemma \ref{lemma:arithmetic-sigma}, the perturbations
    \begin{equation}
        \label{e:w}
        w_{q+1} := w^{(p)}_{q+1}  + d^{(l)} _{q+1} + w^{(c)}_{q+1} 
    \end{equation}
    and
    \begin{equation}
        \label{e:d}
        d_{q+1} := d^{(p)}_{q+1}  + d^{(l)} _{q+1} + d^{(c)}_{q+1} 
    \end{equation}
    are divergence-free, and there exists a dyadic integer $\sigma>50\tau^3$ so that
    \begin{equation}
        \label{e:supp-hatw}
        \supp \widehat{w}_{q+1} \subset \left\{  k \in \mathbb{Z}^3:\frac{1}{2} \sigma < |k| < \frac{9}{10} \sigma \right\};
    \end{equation}
     \begin{equation}
        \label{e:supp-hatd}
        \supp \widehat{d}_{q+1} \subset \left\{ k \in \mathbb{Z}^3:\frac{1}{2} \sigma < |k| < \frac{9}{10} \sigma \right\}.
    \end{equation}
\end{lemma}

\begin{proof}
    We first check that $w_{q+1}$ and $d_{q+1}$ are divergence-free. By using the fact that $(\xi_i \cdot \nabla) \phi_{(\xi)} = 0$ $(i=1,2)$ and $\xi_i \cdot \xi = 0$ $(i=1,2)$, we have $(\xi_i \cdot \nabla) \Phi_{(\xi)} = 0$ $(i=1,2)$. So
    \[ \div {W^{(c)}_{(\xi)}} = \div(\nabla \Delta^{-1} \Phi_{(\xi)}) \xi_1 - (\xi_1 \cdot \nabla) \nabla \Delta^{-1} \Phi_{(\xi)} = \Phi_{(\xi)} \xi_1 - 0 = \tilde{W}_{(\xi)}. \]
    \[ \div {D^{(c)}_{(\xi)}} = \div(\nabla \Delta^{-1} \Phi_{(\xi)}) \xi_2 - (\xi_2 \cdot \nabla) \nabla \Delta^{-1} \Phi_{(\xi)} = \Phi_{(\xi)} \xi_2 - 0 = \tilde{D}_{(\xi)}. \]
    Here, we use the fact that $\Phi_{(\xi)}$ is of mean zero, because $\supp \widehat{\Phi}_{(\xi)} \subset B(\sigma_\xi {\xi},\tau) \cup B(-\sigma_\xi {\xi},\tau)$ is away from 0. Thus, we get
    \[ w_{q+1} = \sum_{\xi \in \Lambda_u\cup\Lambda_B} P_{\le \tau} (a_{(\xi)})\tilde{W}_{(\xi)} +  {W^{(c)}_{(\xi)}} \nabla P_{\le \tau} (a_{(\xi)})= \div \sum_{\xi \in \Lambda_u\cup\Lambda_B}  P_{\le \tau} (a_ {(\xi)})  W_{(\xi)}^{(c)}, \]
    \[ d_{q+1} = \sum_{\xi \in \Lambda_B} P_{\le \tau} (a_{(\xi)}) \tilde{D}_{(\xi)}+ {D^{(c)}_{(\xi)}} \nabla P_{\le \tau} (a_{(\xi)}) = \div \sum_{\xi\in\Lambda_B}  P_{\le \tau} (a_ {(\xi)})  D_{(\xi)}^{(c)}, \]
    Since $P_{\le \tau} (a_ {(\xi)})  W_{(\xi)}^{(c)}$ and $P_{\le \tau} (a_ {(\xi)})  D_{(\xi)}^{(c)}$ is skew-symmetric, we get $\div \ w_{q+1}=0$ and 
    $\div\  d_{q+1}=0$ directly.

    Next, to prove \eqref{e:supp-hatw} -- \eqref{e:supp-hatd}, we need to first show that
    \[
    \supp \widehat{w}_{q+1}\subset \bigcup_{\xi \in \Lambda_u \cup \Lambda_B}\Big( B(\sigma_\xi \xi,2\tau)\cup B(-\sigma_\xi \xi,2\tau)\Big) =: \mathbb{O}.
    \]
    In fact, we know that the Fourier transform of $w^{(p)} + w^{(l)}$ is supported in $\mathbb{O}$. Moreover, $\widehat{\Phi}_{(\xi)}$ is supported in $B(\sigma_\xi \xi, \tau) \cup B(-\sigma_\xi \xi, \tau)$, so $\widehat{W_{(\xi^)}^{(c)}}$ has the same support. 
    It is clear that the Fourier transform of $w^{(c)}_{q+1}$ is supported in $\mathbb{O}$, so $w_{q+1}$ has the same support. Then by using Lemma \ref{lemma:arithmetic-sigma}, we know
   \[ \supp \widehat{w}_{q+1} \subset \left\{ k \in \mathbb{Z}^3:\frac{1}{2} \sigma < |k| < \frac{9}{10} \sigma \right\}.\]
    It is similar with the support of $\widehat{d}_{q+1}$. This completes the proof.
\end{proof}

\subsection{Estimates on the perturbation}
In this part, we aim to verify the inductive estimates \eqref{mitr1} -- \eqref{mitr2} for the velocity and magnetic fields. First, we introduce
the standard Lemma \ref{frecencylemma}.
\begin{lemma}[\cite{CH26}\textit{Lemma 7.1}]\label{frecencylemma}
       Let $1 \le p \le \infty$. For any $f \in C^\infty(\mathbb{T}^d)$ and any dyadic integer $N$, it holds that
    \begin{align*}
        \|P_{\le N} f\|_{L^p} &\lesssim \|f\|_{L^p}, \\
        \|P_{>N} f\|_{L^p} &\lesssim N^{-1} \|\nabla f\|_{L^p}.
    \end{align*}
\end{lemma}

\begin{lemma}[Estimates on $w_{q+1}$]
	\label{esti-wq}
    For $1 \le p \le \infty$, we have the following estimates
    \begin{align}
        \label{e:wp-est}
        \|w_{q+1}^{(p)}\|_{L^p} &\lesssim \ell^{-2\alpha}\rs^{\frac{1}{p}-\frac{1}{2}}, \\
        \label{e:wl-est}
        \|w_{q+1}^{(l)}\|_{L^\infty} &\lesssim \ell^{-2\alpha}\tau^{-1} \lambda \rs^{-\frac{1}{2}}, \\
        \label{e:wc-est}
        \|w_{q+1}^{(c)}\|_{L^\infty} &\lesssim  \ell^{-2\alpha} \sigma^{\delta-1} \rs^{-\frac{1}{2}}, \\
        \label{e:w-est-besov0}
        \|w_{q+1}\|_{B^{0}_{\infty,1}} &\lesssim  \ell^{-2\alpha} \sigma^{\delta} \rs^{-\frac{1}{2}}, \\
        \label{e:w-est-besov}
        \|w_{q+1}\|_{B^{-\theta}_{\infty,1}} &\lesssim \ell^{-2\alpha}\sigma^{ \delta - \theta} \rs^{-\frac{1}{2}},\\
		\label{e:dw-est-besov}
		\|w_{q+1}\|_{C^N} & \lesssim \ell^{-2\alpha}\sigma^N\rs^{-\frac{1}{2}}, \quad 1\le N \le 7.
    \end{align}
\end{lemma}
\begin{proof}
    First, for \eqref{e:wp-est}, we use H\"older's inequality directly as
    \begin{align*}
        \|w_{q+1}^{(p)}\|_{L^p}
        &\le \sum_{\xi \in \Lambda_u\cup\Lambda_B} \|a_{(\xi)}\|_{L^\infty} \|W_{(\xi)}\|_{L^p} \|\cos(2\pi \sigma_{\xi} {\xs} \cdot x)\|_{L^\infty} \\
        &\le \sum_{\xi \in \Lambda_u\cup\Lambda_B} \|a_{(\xi)}\|_{L^\infty} \|W_{(\xi)}\|_{L^p} \lesssim  \ell^{-2\alpha}\rs^{\frac{1}{p}-\frac{1}{2}}.
    \end{align*}
    Next, for \eqref{e:wl-est}, we follow from Lemma \ref{frecencylemma} as
    \begin{align*}
        \|a_{(\xi)} P_{>\tau}(W_{(\xi)}) \cos(2\pi \sigma_{\xi} {\xs}\cdot x)\|_{L^\infty}
        &\le \|a_{(\xi)}\|_{L^\infty} \|P_{>\tau}(W_{(\xi)})\|_{L^\infty} \\
        &\lesssim \|a_{(\xi)}\|_{L^\infty} \tau^{-1} \|\nabla (W_{(\xi)})\|_{L^\infty} \\
        &\lesssim \ell^{-2\alpha}\tau^{-1} \lambda \rs^{-\frac{1}{2}}.
    \end{align*}
    Moreover, we have
    \begin{align*}
        \|P_{>\tau} (a_{(\xi)})P_{\le \tau}(W_{(\xi)}) \cos(2\pi \sigma_{\xi} \xs \cdot x)\|_{L^\infty}
        &\lesssim \|P_{>\tau} (a_{(\xi)})\|_{L^\infty} \|W_{(\xi)}\|_{L^\infty} \\
        &\lesssim \tau^{-1} \|a_{(\xi)}\|_{C^1} \rs^{-\frac{1}{2}} \\
        &\lesssim \ell^{-2 \alpha}\tau^{-1} \rs^{-\frac{1}{2}}.
    \end{align*}
    For \eqref{e:wc-est}, we can use Bernstein's inequality to get
    \begin{align*}
        \|W_{(\xi)}^{(c)}\|_{L^\infty} 
        &\lesssim \|\nabla \Delta^{-1} \Phi_{(\xi)}\|_{C^\delta} \lesssim \sigma^{-1} \|\Phi_{(\xi)}\|_{C^\delta} \\
        &\le \sigma^{-1} \|P_{\le \tau} (\phi_{(\xi)}) \cos(2\pi \sigma_{\xi} \xs \cdot x)\|_{C^\delta} \lesssim \sigma^{\delta-1}  \rs^{-\frac{1}{2}}.
    \end{align*}
    Here, we also use the support of $\widehat{\Phi}_{(\xi)}$ which lies in the annulus $\{k\in \mathbb{Z}^3:\sigma/2 < |k| < 9\sigma/10\}$. Thus, we have
    \[ \|w_{q+1}^{(c)}\|_{L^\infty} \le \sum_{\xi \in \Lambda_u \cup \Lambda_B} \|\nabla P_{\le \tau} (a_{(\xi)})\|_{L^\infty} \|W_{(\xi)}^{(c)}\|_{L^\infty} \lesssim \ell^{-2 \alpha}\sigma^{\delta-1} \rs^{-\frac{1}{2}}.\label{e:w(p)+w(l)} \]

    Finally, to prove \eqref{e:w-est-besov0}, \eqref{e:w-est-besov} and \eqref{e:dw-est-besov}, 
    since \[ w_{q+1} = \sum_{\xi \in \Lambda_u\cup\Lambda_B} P_{\le \tau} (a_{(\xi)}) \tilde{W}_{(\xi)} +  \nabla P_{\le \tau} (a_{(\xi)}) {W^{(c)}_{(\xi)}}= \div \sum_{\xi \in \Lambda_u\cup\Lambda_B}  P_{\le \tau} (a_ {(\xi)})  W_{(\xi)}^{(c)}, \]
    we have \eqref{e:dw-est-besov} as 
	 \begin{align*}
        \|w_{q+1}\|_{C^N} 
        &\lesssim \|\nabla^{N+1} P_{\le \tau} (a_{(\xi)})\|_{L^\infty} \|W_{(\xi)}^{(c)}\|_{L^\infty} 
		+\|P_{\le \tau} (a_{(\xi)})\|_{L^\infty} \|\nabla^{N+1} W_{(\xi)}^{(c)}\|_{L^\infty}\\ 
        &\lesssim \ell^{-2(N+1)\alpha}\lambda^{-1}\rs^{-\frac{1}{2}}+\ell^{-2\alpha}\sigma^N\rs^{-\frac{1}{2}} \lesssim \ell^{-2 \alpha}\sigma^N\rs^{-\frac{1}{2}}.
    \end{align*}
    So, by using  \eqref{e:supp-hatw}, \eqref{e:wp-est} -- \eqref{e:wc-est} and \eqref{e:dw-est-besov}, we can get
    \begin{align*}
         \|w_{q+1}\|_{B^{0}_{\infty,1}}  \lesssim \ln\sigma \|w_{q+1}\|_{B^{0}_{\infty,\infty}} \lesssim \ln\sigma \|w_{q+1}\|_{C^{\delta/2}}\lesssim \ell^{-2\alpha} \sigma^{\delta} \rs^{-\frac{1}{2}}.
    \end{align*}   

    Since $\widehat{w}_{q+1}$ is supported in the annulus $\{k \in \mathbb{Z}^3:\sigma/2 < |k| < 9\sigma/10\}$ by using \eqref{e:w-est-besov0}, Bernstein inequality and embedding theorems, we can get
    \begin{align*}
        \|w_{q+1}\|_{B^{-\theta}_{\infty,1}} 
        &\lesssim \sigma^{-\theta} \|w_{q+1}\|_{B^{0}_{\infty,1}} \le \ell^{-2\alpha} \sigma^{\delta -\theta} \rs^{-\frac{1}{2}}.
    \end{align*}
	
    This completes the proof.
\end{proof}
Simlarly, we can get the same estimates on $d_{q+1}$.
\begin{lemma}[Estimates on $d_{q+1}$]
	\label{esti-dq}
    For $1 \le p \le \infty$, we have the following estimates
    \begin{align}
        \label{e:dp-est}
        \|d_{q+1}^{(p)}\|_{L^p} &\lesssim \ell^{-\alpha}\rs^{\frac{1}{p}-\frac{1}{2}}, \\
        \label{e:dl-est}
        \|d_{q+1}^{(l)}\|_{L^\infty} &\lesssim \ell^{-\alpha}\tau^{-1} \lambda \rs^{-\frac{1}{2}}, \\
        \label{e:dc-est}
        \|d_{q+1}^{(c)}\|_{L^\infty} &\lesssim  \ell^{-\alpha}\sigma^{\delta-1} \rs^{-\frac{1}{2}}, \\
        \label{e:d-est-besov0}
        \|d_{q+1}\|_{B^{0}_{\infty,1}} &\lesssim  \ell^{-\alpha } \sigma^{\delta} \rs^{-\frac{1}{2}}, \\
        \label{e:d-est-besov}
        \|d_{q+1}\|_{B^{-\theta}_{\infty,1}} &\lesssim \ell^{-\alpha}\sigma^{\delta -\theta} \rs^{-\frac{1}{2}},\\
        \label{e:dd-est-besov}
		\|d_{q+1}\|_{C^N} & \lesssim \ell^{-\alpha }\sigma^N\rs^{-\frac{1}{2}}, \quad 1 \le N \le 7.
    \end{align}
\end{lemma}
\section{Reynolds stress}
\subsection{Inverse divergence}
 In order to define the Reynolds and magnetic stresses,
 we can use the inverse-divergence operator that acts on mean-free vector fields from~\cite{bdis15}:
\begin{equation*}
 (\mathcal{R}^u v)^{kl} = \partial_k \Delta^{-1} v^l + \partial_l \Delta^{-1} v^k - \frac{1}{2}(\delta_{kl} + \partial_k \partial_l \Delta^{-1})\div \Delta^{-1} v, 
\end{equation*}
where $k, \ell \in \{1,2,3 \}$. The operator $\mathcal{R}^u$ returns a symmetric, trace-free matrix 
and satisfies the following key identity for mean-free vector fields: $\div \mathcal{R}^u(v) = v$. Note that $|\nabla| \mathcal{R}^u$ is a Calderon-Zygmund operator. 

Moreover, we will also need an inverse divergence that returns skew-symmetric matrices. We want 
$\div \mathcal{R}^B (f) = f$ with $\div f = 0$ where $f: \mathbb{R}^3 \to \mathbb{R}^3$ and $\mathcal{R}^B(f) = - (\mathcal{R}^B(f))^{\top}$. As in \cite{bbv20}, we can define
\begin{equation*}
(\mathcal{R}^Bf)_{ij} :=  \varepsilon_{ijk} (-\Delta)^{-1}(\curl f)_k,
\end{equation*}
where $\varepsilon_{ijk}$ is the Levi-Civita tensor. Note that $|\nabla| \mathcal{R}^u$ is also a Calderon-Zygmund
operator.
\subsection{Decomposition of magnetic and Reynolds stresses.}
Here, we intend to decompose the magnetic and Reynolds stresses. 
Let us first consider the magnetic stress. By replacing $q$ with $q+1$ in the equations \eqref{R_MHDs}, we derive that
\begin{align}
	\displaystyle\div\mathring{R}_{q+1}^B
	&=\underbrace{\nu_2(-\Delta)^{\alpha_2} d_{q+1}  +\div (d_{q + 1} \otimes u_{q}- u_{q} \otimes d_{q+1}+ B_{q}\otimes w_{q+1} -w_{q +1} \otimes B_{q} )}_{ \div\mathring{R}_{\mathrm{lin}}^B  }   \notag\\
	&\quad+\underbrace{\div (d_{q+ 1}^{(p)} \otimes w_{q+1}^{(p)} -w_{q+ 1}^{(p)} \otimes  d_{q+ 1}^{(p)} + \rb)}_{\div\mathring{R}_{{\mathrm{osc}}}^B }  \notag\\
	&\quad+\div\Big( d_{q+1}^{(p)} \otimes (w_{q+1}^{(c)}+w_{q+1}^{(l)}) -(w_{q+1}^{(c)} +w_{q+1}^{(l)}) \otimes d_{q+1}  \notag  \\
	&\qquad \underbrace{\qquad\quad+(d_{q+1}^{(c)}+d_{q+1}^{(l)})\otimes w_{q+1}-w_{q+1}^{(p)} \otimes (d_{q+1}^{(c)} +d_{q+1}^{(l)})\Big) }_{\div\mathring{R}_{\mathrm{cor}}^B }. \label{rb.2}
\end{align}

Using the inverse-divergence operator $\mathcal{R}^B$, we can define $\mathring{R}_{q+1}^{B}$ as
\[
\mathring{R}_{q+1}^{B}=\mathring{R}_{\mathrm{lin}}^B+\mathring{R}_{{\mathrm{osc}}}^B+\mathring{R}_{\mathrm{cor}}^B,
\]
where
\begin{align*}
\mathring{R}_{\mathrm{lin}}^B&=\mathcal{R}^B\(\nu_2(-\Delta)^{\alpha_2} d_{q+1} \) +d_{q + 1} \otimes u_{q}- u_{q} \otimes d_{q+1}+ B_{q}\otimes w_{q+1} -w_{q +1} \otimes B_{q},\\
\mathring{R}_{{\mathrm{osc}}}^B &= \sum_{\xi \in \Lambda_B} \frac{1}{2} a_{(\xi)}^2  \P_{> 0}(D_{(\xi)}\otimes W_{(\xi)}-W_{(\xi)}\otimes D_{(\xi)}) \nonumber\\
&+ \sum_{\xi \in \Lambda_B} \frac{1}{2} \cos(4\pi \sigma_{\xi}{\xs} \cdot x)  a_{(\xi)}^2 (D_{(\xi)}\otimes W_{(\xi)}-W_{(\xi)}\otimes D_{(\xi)} )\nonumber\\
&+ \(\sum_{\xi \neq \xi' \in \Lambda_B}+ \sum_{\xi \in \Lambda_u,\xi' \in \Lambda_B}\)a_{(\xi)}a_{(\xi')}\cos(2\pi \sigma_{\xi}{\xs} \cdot x) \cos(2\pi \sigma_{\xi'}{\xi^{'}} \cdot x) \nonumber\\
&(D_{(\xi')}\otimes W_{(\xi)}-W_{(\xi)}\otimes D_{(\xi^{'})})\nonumber\\
&:=\mathring{R}_{{\mathrm{osc}},1}^B +\mathring{R}_{{\mathrm{osc}},2}^B +\mathring{R}_{{\mathrm{osc}},3}^B,\\
\mathring{R}_{\mathrm{cor}}^B&=d_{q+1}^{(p)} \otimes (w_{q+1}^{(c)}+w_{q+1}^{(l)}) -(w_{q+1}^{(c)} +w_{q+1}^{(l)}) \otimes d_{q+1} \\
&+(d_{q+1}^{(c)}+d_{q+1}^{(l)}) \otimes w_{q+1}-w_{q+1}^{(p)} {\otimes}  (d_{q+1}^{(c)} +d_{q+1}^{(l)}).
\end{align*}
{\noindent\bf $\bullet$ Decomposition of Reynolds stresses.}
Moreover, we also cosider the Reynolds stress:
\begin{align}
	\displaystyle\div\mathring{R}_{q+1}^u
	&\displaystyle = \underbrace{ \nu_1(-\Delta)^{\alpha_1} w_{q+1} +\div\big(u_{q} \otimes w_{q+1} + w_{q+ 1} \otimes u_{q} - B_{q} \otimes d_{q+1} - d_{q+1} \otimes B_{q}\big) }_{ \div\mathring{R}_{\mathrm{lin}}^u +\nabla P_{\mathrm{lin}} }   \notag\\
	&\displaystyle\quad+ \underbrace{\div (w_{q+1}^{(p)} \otimes w_{q+1}^{(p)} - d_{q+1}^{(p)} \otimes d_{q+1}^{(p)} +  \ru)}_{\div\mathring{R}^u_{{\mathrm{osc}}} +\nabla P_{{\mathrm{osc}}}}  \notag\\
	&\displaystyle\quad+\div\Big((w_{q+1}^{(c)}+w_{q+1}^{(l)})\otimes w_{q+1}+ w_{q+1}^{(p)} \otimes (w_{q+1}^{(c)}+w_{q+1}^{(l)})  \notag \\
	&\qquad \underbrace{\qquad - (d_{q+1}^{(c)}+d_{q+1}^{(l)})\otimes d_{q+1}
		- d_{q+1}^{(p)} \otimes (d_{q+1}^{(c)}+d_{q+1}^{(l)}) \Big)}_{\div\mathring{R}_{\mathrm{cor}}^u +\nabla P_{\mathrm{cor}}}\label{ru.2}\notag\\
	&+\nabla P_{q+1}-\nabla P_{q}. 
\end{align}

Using the inverse-divergence operator $\mathcal{R}^u$,we can define $\mathring{R}_{q+1}^{u}$:
\[
\mathring{R}_{q+1}^{u}=\mathring{R}_{\mathrm{lin}}^u+\mathring{R}^u_{{\mathrm{osc}}}+\mathring{R}_{\mathrm{cor}}^u,
\]
where
\begin{align*}
\mathring{R}_{\mathrm{lin}}^u &= \mathcal{R}^u \left( \nu_1(-\Delta)^{\alpha_1} w_{q+1} \right) + u_{q} \rotimes w_{q+1} + w_{q+ 1} \rotimes u_{q} - B_{q} \rotimes d_{q+1} - d_{q+1} \rotimes B_{q}, \\
\mathring{R}_{{\mathrm{osc}}}^u &= \underbrace{\sum_{\xi \in \Lambda_u} \frac{1}{2} a_{(\xi)}^2 \P_{>0}( W_{(\xi)}\rotimes W_{(\xi)}) + \sum_{\xi \in \Lambda_B} \frac{1}{2} a_{(\xi)}^2 \P_{>0} (W_{(\xi)}\rotimes W_{(\xi)}-D_{(\xi)}\rotimes D_{(\xi)})}_{\mathring{R}_{{\mathrm{osc}},1}^u} \nonumber \\
& \quad + \sum_{\xi \in \Lambda_u} \frac{1}{2} \cos(4\pi \sigma_{\xi}{\xs} \cdot x) a_{(\xi)}^2 ( W_{(\xi)}\rotimes W_{(\xi)}) \nonumber \\
& \quad \underbrace{+ \sum_{\xi \in \Lambda_B} \frac{1}{2} \cos(4\pi \sigma_{\xi}{\xs} \cdot x) a_{(\xi)}^2 (W_{(\xi)}\rotimes W_{(\xi)}-D_{(\xi)}\rotimes D_{(\xi)})}_{\mathring{R}_{{\mathrm{osc}},2}^u} \nonumber \\
& \quad + \sum_{\xi \neq \xi' \in \Lambda_u\cup\Lambda_B }a_{(\xi)}a_{(\xi')} \cos(2\pi \sigma_{\xi}{\xs} \cdot x) \cos(2\pi \sigma_{\xi'}{\xi^{'}} \cdot x) W_{(\xi)}\rotimes W_{(\xi')}\nonumber \\
& \quad \underbrace{-\sum_{\xi \neq \xi' \in \Lambda_B}a_{(\xi)}a_{(\xi')}\cos(2\pi \sigma_{\xi}{\xs} \cdot x) \cos(2\pi \sigma_{\xi'}{\xi^{'}} \cdot x)D_{(\xi)}\rotimes D_{(\xi')}}_{\mathring{R}_{{\mathrm{osc}},3}^u}, \\
\mathring{R}_{\mathrm{cor}}^u &= (w_{q+1}^{(c)}+w_{q+1}^{(l)})\mathring{\otimes} w_{q+1}+ w_{q+1}^{(p)} \mathring{\otimes} (w_{q+1}^{(c)}+w_{q+1}^{(l)}) \nonumber\\
& \quad - (d_{q+1}^{(c)}+d_{q+1}^{(l)})\mathring{\otimes} d_{q+1}- d_{q+1}^{(p)} \mathring{\otimes} (d_{q+1}^{(c)}+d_{q+1}^{(l)}).
\end{align*}
\subsection{$\Hs$-Estimates for the magnetic stress and Reynolds stress.}
\begin{lemma}[$\Hs$-Estimates on $\mathring{R}_{q+1}^B$]
    \label{lemma:rb}
    For $s \ge \frac{3}{2}+2\alpha_2$, we have the following estimates
    \begin{align}
        \label{b:Rosc-est}
        &\|\mathring{R}_{\mathrm{osc},1}^{B}\|_{\Hs} \lesssim  \ell^{-2 \alpha }(\lambda\rs)^{-s} \rs^{-\frac{1}{2}} + \ell^{-2\alpha}(\lambda\rs)^{-1} , \\
        \label{b:Rdis-est}
        &\|\mathring{R}_{\mathrm{osc},2}^{B}\|_{\Hs} \lesssim \ell^{-2\alpha}\tau^{-1} \lambda+\ell^{-2\alpha}\sigma^{-s} \rs^{-\frac{1}{2}}, \\
        \label{b:Roff-est}
        &\|\mathring{R}_{\mathrm{osc},3}^{B}\|_{L^1} \lesssim \ell^{-3\alpha}\rs, \\
        \label{b:Rcor-est}
        &\|\mathring{R}_{\mathrm{cor}}^{B}\|_{L^1} \lesssim \ell^{-3\alpha}\tau^{-1} \lambda+ \ell^{-3 \alpha}\sigma^{\delta-1}, \\
        \label{b:Rlin-est}
        &\|\mathring{R}_{\mathrm{lin}}^{B}\|_{\Hs} \lesssim  \ell^{-\alpha} \rs^{\frac{1}{2}}.
    \end{align}
\end{lemma}

\begin{proof}
    For \eqref{b:Rosc-est}, we consider to decompose $\mathring{R}_{{\mathrm{osc}},1}^{B}$ as
    \begin{align*} 
	\mathring{R}_{{\mathrm{osc}},1}^{B} &= \frac{1}{2} \sum_{\xi \in \Lambda_B} a_{(\xi)}^2 \P_{> 0}(D_{(\xi)}\otimes W_{(\xi)}-W_{(\xi)}\otimes D_{(\xi)})\notag\\
	&= \frac{1}{2} \sum_{\xi \in \Lambda_B} \big(P_{\le (\lambda\rs)/2} (a_{(\xi)}^2)\P_{> 0}(D_{(\xi)}\otimes W_{(\xi)}-W_{(\xi)}\otimes D_{(\xi)}) \notag \\
	&+ P_{>(\lambda\rs)/2} (a_{(\xi)}^2) \P_{> 0}(D_{(\xi)}\otimes W_{(\xi)}-W_{(\xi)}\otimes D_{(\xi)}) \big).\notag \
	\end{align*}
	Define $T_{(\xi)}:=\P_{> 0}(D_{(\xi)}\otimes W_{(\xi)}-W_{(\xi)}\otimes D_{(\xi)}) $,
	since $\hat{T}_{(\xi)}$ is supported in $\{k \in \mathbb{Z}^3:|k| \ge \lambda\rs\}$, we can get
    \begin{align*}
        \|P_{\le (\lambda\rs)/2} (a_{(\xi)}^2) T_{\xi}\|_{\Hs} 
        &\lesssim (\lambda\rs)^{-s} \|P_{\le (\lambda\rs)/2} (a_{(\xi)}^2) T_{(\xi)}\|_{L^2} \\
        &\lesssim(\lambda\rs)^{-s} \|P_{\le (\lambda\rs)/2} (a_{(\xi)}^2)\|_{L^\infty} \|T_{(\xi)}\|_{L^2} \\
        &\lesssim (\lambda\rs)^{-s} \|a_{(\xi)}\|_{L^\infty}^2 \|W_{(\xi)}\|_{L^4}\|D_{(\xi)}\|_{L^4} \lesssim \ell^{-2\alpha}(\lambda\rs)^{-s} \rs^{-\frac{1}{2}},
    \end{align*}
    Using Lemma \ref{frecencylemma}, we can get 
    \begin{align*}
        \|P_{>(\lambda\rs)/2} (a_{(\xi)}^2) T_{(\xi)}\|_{L^1} 
        &\lesssim \|P_{>(\lambda\rs)/2} (a_{(\xi)}^2)\|_{L^\infty} \|T_{(\xi)}\|_{L^1} \\
        &\lesssim (\lambda\rs)^{-1} \|\nabla(a_{(\xi)}^2)\|_{L^\infty} \|W_{(\xi)}\|_{L^2} \|D_{(\xi)}\|_{L^2} \lesssim \ell^{-2\alpha}(\lambda\rs)^{-1}.
    \end{align*}
    Since $-s<-3/2$, by Sobolev's embedding theorem, we know that $L^1$ embeds into $\Hs$. So these estimates above give \eqref{b:Rosc-est}.

    Next, to get \eqref{b:Rdis-est}, we decompose $\mathring{R}_{{\mathrm{osc}},2}^{B}$ into three parts
    \begin{align*}
        \mathring{R}_{{\mathrm{osc}},2}^{B}
        &:=\frac{1}{2}\sum_{\xi \in \Lambda_B} P_{\le \tau}(a_{(\xi)}^2) P_{\le \tau} (D_{(\xi)}\otimes W_{(\xi)}-W_{(\xi)}\otimes D_{(\xi)}) \cos(4\pi \sigma_{\xi} {\xi} \cdot x) \\
        &\quad +\frac{1}{2}\sum_{\xi \in \Lambda_B} P_{>\tau}(a_{(\xi)}^2) P_{\le \tau} (D_{(\xi)}\otimes W_{(\xi)}-W_{(\xi)}\otimes D_{(\xi)}) \cos(4\pi \sigma_{\xi} {\xi} \cdot x) \\   
		&\quad +\frac{1}{2}\sum_{\xi \in \Lambda_B}(a_{(\xi)}^2)  P_{> \tau}(D_{(\xi)}\otimes W_{(\xi)}-W_{(\xi)}\otimes D_{(\xi)}) \cos(4\pi \sigma_{\xi} {\xi} \cdot x)\\
        &:=I_1 + I_2 + I_3
    \end{align*}
    Noticing that $\hat{I}_1$ is supported in $\{k \in \mathbb{Z}^3:\sigma<|k|<9\sigma/5\}$, we have
    \begin{align*}
        \|I_1\|_{\Hs} 
        &\leq \sigma^{-s} \sum_{\xi \in \Lambda_B} \|P_{\le \tau} (a_{(\xi)}^2)\|_{L^\infty} \|P_{\le \tau} (D_{(\xi)}\otimes W_{(\xi)}-W_{(\xi)}\otimes D_{(\xi)})\|_{L^2} \\
        &\leq \sigma^{-s} \sum_{\xi \in \Lambda_B} \|a_{(\xi)}\|_{L^\infty}^2 \|W_{(\xi)}\|_{L^4}\|D_{(\xi)}\|_{L^4} \lesssim \ell^{-2\alpha}\sigma^{-s} \rs^{-\frac{1}{2}}.
    \end{align*}
    The estimates of $I_2$ and $I_3$ can use Lemma \ref{frecencylemma} to get
    \begin{align*}
        \|I_2\|_{L^1} 
        &\leq\sum_{\xi \in \Lambda_B} \|P_{>\tau}(a_{(\xi)}^2)\|_{L^\infty} \|P_{\le \tau} (D_{(\xi)}\otimes W_{(\xi)}-W_{(\xi)}\otimes D_{(\xi)})\|_{L^1} \\
        &\leq\tau^{-1} \sum_{\xi \in \Lambda_B} \|\nabla (a_{(\xi)}^2)\|_{L^\infty} \|W_{(\xi)}\|_{L^2} \|D_{(\xi)}\|_{L^2} \lesssim \ell^{-2\alpha}\tau^{-1},
    \end{align*}
    and
    \begin{align*}
        \|I_3\|_{L^1} 
        &\leq \sum_{\xi \in \Lambda_B} \|a_{(\xi)}^2\|_{L^\infty} \|P_{>\tau} (D_{(\xi)}\otimes W_{(\xi)}-W_{(\xi)}\otimes D_{(\xi)})\|_{L^1} \\
        &\leq \tau^{-1} \sum_{\xi \in \Lambda_B} \|a_{(\xi)}\|_{L^\infty}^2 \|\nabla (D_{(\xi)}\otimes W_{(\xi)}-W_{(\xi)}\otimes D_{(\xi)})\|_{L^1} \\
        &\leq \tau^{-1} \sum_{\xi \in \Lambda_B} \|a_{(\xi)}\|_{L^\infty}^2 (\|\nabla (W_{(\xi)})\|_{L^2} \|D_{(\xi)}\|_{L^2}+ \|\nabla (D_{(\xi)})\|_{L^2} \|W_{(\xi)}\|_{L^2})\\
		&\lesssim\ell^{-2\alpha}\tau^{-1} \lambda.
    \end{align*}
    Using $L^1$ embedding into $\Hs$ again, we obtain the estimate desired in \eqref{b:Rdis-est}.

    For $\mathring{R}_{{\mathrm{osc}},3}^{B}$, we can use H\"older's inequality and Lemma \ref{buildingblockestlemma} as
    \[ \|\mathring{R}_{{\mathrm{osc}},3}^{B}\|_{L^1} \lesssim 
	\|a_{(\xi)}\|_{L^\infty} \|a_{(\xi')}\|_{L^\infty} \max_{\xi\ne \xi'} \|D_{(\xi')} \otimes W_{(\xi)}\|_{L^1} \lesssim \ell^{-3\alpha}\rs. \]

    The estimate of $\mathring{R}_{\mathrm{cor}}^{B}$ directly follows from H\"older's inequality.
	\begin{align*}
	\|\mathring{R}_{\mathrm{cor}}^{B}\|_{L^1} &\leq
	\|d_{q+1}^{(p)}\|_{L^1} (\|w_{q+1}^{(c)}\|_{L^\infty} + \|w_{q+1}^{(l)}\|_{L^\infty})\\
	&+\|d_{q+1}\|_{L^1} (\|w_{q+1}^{(c)}\|_{L^\infty} + \|w_{q+1}^{(l)}\|_{L^\infty})\\
	&+\|w_{q+1}^{(p)}\|_{L^1} (\|d_{q+1}^{(c)}\|_{L^\infty} + \|d_{q+1}^{(l)}\|_{L^\infty})\\
	&+\|w_{q+1}\|_{L^1} (\|d_{q+1}^{(c)}\|_{L^\infty} + \|d_{q+1}^{(l)}\|_{L^\infty}).
	\end{align*}
	So by \eqref{e:wp-est} -- \eqref{e:wc-est} and \eqref{e:dp-est} -- \eqref{e:dc-est}, 
    we can get$$ \|\mathring{R}_{\mathrm{cor}}^{B}\|_{L^1} \lesssim \ell^{-3\alpha}\tau^{-1} \lambda+ \ell^{-3\alpha}\sigma^{\delta-1}.$$

    For $\mathring{R}_{\mathrm{lin}}^{B}$, first we can get estimates by using H\"older's inequality, 
    \begin{align*}
       &\|d_{q + 1} \otimes u_{q} - u_{q} \otimes d_{q+1}+ B_{q}\otimes w_{q+1} -w_{q +1} \otimes B_{q} \|_{L^1}\\
        &\lesssim \|B_{q}\|_{L^\infty} \|w_{q+1}\|_{L^1} + \|u_{q}\|_{L^\infty}\|d_{q+1}\|_{L^1}\\
		&\lesssim \ell^{-3 \alpha}\rs^{\frac{1}{2}}.
    \end{align*}
    Moreover, as $-s+2\alpha_2-1 < -3/2$, Sobolev's embedding theorem yields that $L^1$ embeds into ${H}^{-s+2\alpha_2-1}$, so we get
    \[ \|\mathcal{R}^B(-\Delta)^{\alpha_2} d_{q+1}\|_{\Hs} \leq \|d_{q+1}\|_{{H}^{-s+2\alpha_2-1}} \leq\|d_{q+1}\|_{L^1} \le\ell^{-\alpha}\rs^{\frac{1}{2}}. \]
	This completes the proof.
\end{proof}
\begin{lemma}[$\Hs$-Estimates on $\mathring{R}_{q+1}^u$]
    \label{lemma:ru}
    Let $s \ge \frac{3}{2}+2\alpha_1$. Then we have the following estimates
    \begin{align}
        \label{e:Rosc-est}
        &\|\mathring{R}_{\mathrm{osc},1}^{u}\|_{\Hs} \lesssim  \ell^{-4\alpha}(\lambda\rs)^{-s} \rs^{-\frac{1}{2}} + \ell^{-4\alpha}(\lambda\rs)^{-1} , \\
        \label{e:Rdis-est}
        &\|\mathring{R}_{\mathrm{osc},2}^{u}\|_{\Hs} \lesssim  \ell^{-4\alpha}\tau^{-1} \lambda+\ell^{-4\alpha}\sigma^{-s} \rs^{-\frac{1}{2}}, \\
        \label{e:Roff-est}
        &\|\mathring{R}_{\mathrm{osc},3}^{u}\|_{L^1} \lesssim \ell^{-4\alpha}\rs, \\
        \label{e:Rcor-est}
        &\|\mathring{R}_{\mathrm{cor}}^{u}\|_{L^1} \lesssim \ell^{-4\alpha}\tau^{-1} \lambda+ \ell^{-4\alpha}\sigma^{\delta-1}, \\
        \label{e:Rlin-est}
        &\|\mathring{R}_{\mathrm{lin}}^{u}\|_{\Hs} \lesssim   \ell^{-2\alpha}\rs^{\frac{1}{2}}.
    \end{align}
\end{lemma}
\begin{proof}
    First, we consider \eqref{e:Rosc-est}. We decompose $\mathring{R}_{\mathrm{osc},1}^{u}$ as
    \begin{align*} 
	\mathring{R}^{u}_{{\mathrm{osc}},1} &= \frac{1}{2} \sum_{\xi \in \Lambda_u} a_{(\xi)}^2 \P_{> 0}(W_{(\xi)}{\rotimes} W_{(\xi)}) 
	+\frac{1}{2} \sum_{\xi \in \Lambda_B} a_{(\xi)}^2 \P_{> 0}(W_{(\xi)}{\rotimes}W_{(\xi)}-D_{(\xi)}{\rotimes} D_{(\xi)} )\\
	&= \frac{1}{2} \sum_{\xi \in \Lambda_u} \left(P_{\le (\lambda\rs)/2} (a_{(\xi)}^2)\P_{> 0}(W_{(\xi)}{\rotimes} W_{(\xi)})  
	+ P_{>(\lambda\rs)/2} (a_{(\xi)}^2) \P_{> 0}(W_{(\xi)}{\rotimes} W_{(\xi)}) \right)\\
	&+\frac{1}{2} \sum_{\xi \in \Lambda_B} \left((P_{\le (\lambda\rs)/2}+ P_{>(\lambda\rs)/2}) (a_{(\xi)}^2)\P_{> 0}(W_{(\xi)}{\rotimes}W_{(\xi)}-D_{(\xi)}{\rotimes} D_{(\xi)}) \right).
	\end{align*}
    Since the norm of $W_{(\xi)}{\rotimes} W_{(\xi)}$ and the norm of $W_{(\xi)}{\otimes} W_{(\xi)}$ 's trace can be directly estimated by $W_{(\xi)}{\otimes} W_{(\xi)}$, 
    so below we only show the calculation for the norm of $W_{(\xi)}{\otimes} W_{(\xi)}$. 

	Denote $S^{1}_{(\xi)}:=\P_{> 0}(W_{(\xi)}{\otimes} W_{(\xi)}) $ and $S^{2}_{(\xi)}:=\P_{> 0}(W_{(\xi)}{\otimes} W_{(\xi)}-D_{(\xi)}{\otimes} D_{(\xi)} ) $.
	Since $\widehat{S^{1}}_{(\xi)}$ and $\widehat{S^{2}}_{(\xi)}$ are supported on $\{k \in \mathbb{Z}^3:|k| \ge \lambda\rs\}$, we can get
    \begin{align*}
        \|P_{\le (\lambda\rs)/2} (a_{(\xi)}^2) S^{1}_{(\xi)}\|_{\Hs} 
        &\lesssim (\lambda\rs)^{-s} \|P_{\le (\lambda\rs)/2} (a_{(\xi)}^2) S^{1}_{(\xi)}\|_{L^2} \\
        &\lesssim (\lambda\rs)^{-s} \|P_{\le (\lambda\rs)/2} (a_{(\xi)}^2)\|_{L^\infty} \|S^{1}_{(\xi)}\|_{L^2} \\
        &\lesssim (\lambda\rs)^{-s} \|a_{(\xi)}\|_{L^\infty}^2 \|W_{(\xi)}\|_{L^4}^2 \lesssim (\lambda\rs)^{-s} \ell^{-4\alpha}\rs^{-\frac{1}{2}}.
    \end{align*}
    By using Lemma \ref{frecencylemma}, we can get 
    \begin{align*}
        \|P_{>(\lambda\rs)/2} (a_{(\xi)}^2) S^{1}_{(\xi)}\|_{L^1} 
        &\lesssim \|P_{>(\lambda\rs)/2} (a_{(\xi)}^2)\|_{L^\infty} \|S^{1}_{(\xi)}\|_{L^1} \\
        &\lesssim (\lambda\rs)^{-1} \|\nabla(a_{(\xi)}^2)\|_{L^\infty} \|W_{(\xi)}\|_{L^2}^2 \lesssim \ell^{-4\alpha}(\lambda\rs)^{-1}.
    \end{align*}
	Since $-s<-3/2$, using Sobolev's embedding theorem, we know that $L^1$ embeds into $\Hs$. 
    Simlarly, we can get:
	\begin{align*}
		&\|(P_{\le (\lambda\rs)/2}+ P_{>(\lambda\rs)/2}) (a_{(\xi)}^2) S^{2}_{(\xi)}\|_{\Hs}\\
		&\lesssim (\lambda\rs)^{-s} \ell^{-4 \alpha}\rs^{-\frac{1}{2}} + \ell^{-4 \alpha}(\lambda\rs)^{-1} 
	\end{align*}
     So these estimates above give \eqref{e:Rosc-est}.

    For \eqref{e:Rdis-est}, we intend to decompose $\mathring{R}_{{\mathrm{osc}},2}^{u}$ into six parts
    \begin{align*}
        \mathring{R}_{{\mathrm{osc}},2}^{u}
        &=\frac{1}{2}\sum_{\xi \in \Lambda_u} P_{\le \tau}(a_{(\xi)}^2) P_{\le \tau} (W_{(\xi)}\rotimes W_{(\xi)}) \cos(4\pi \sigma_{\xi} {\xi} \cdot x) \\
        &\quad +\frac{1}{2}\sum_{\xi \in \Lambda_u} P_{>\tau}(a_{(\xi)}^2) P_{\le \tau} (W_{(\xi)} \rotimes W_{(\xi)}) \cos(4\pi \sigma_{\xi} {\xi} \cdot x) \\   
		&\quad +\frac{1}{2}\sum_{\xi \in \Lambda_u}(a_{(\xi)}^2)  P_{> \tau}(W_{(\xi)}\rotimes W_{(\xi)}) \cos(4\pi \sigma_{\xi} {\xi} \cdot x)\\
		&\quad +\frac{1}{2}\sum_{\xi \in \Lambda_B} P_{\le \tau}(a_{(\xi)}^2) P_{\le \tau} (W_{(\xi)}\rotimes W_{(\xi)}-D_{(\xi)}\rotimes D_{(\xi)}) \cos(4\pi \sigma_{\xi} {\xi} \cdot x) \\
        &\quad +\frac{1}{2}\sum_{\xi \in \Lambda_B} P_{>\tau}(a_{(\xi)}^2) P_{\le \tau} (W_{(\xi)} \rotimes W_{(\xi)}-D_{(\xi)}\rotimes D_{(\xi)}) \cos(4\pi \sigma_{\xi} {\xi} \cdot x) \\   
		&\quad +\frac{1}{2}\sum_{\xi \in \Lambda_B}(a_{(\xi)}^2)  P_{> \tau}(W_{(\xi)} \rotimes W_{(\xi)}-D_{(\xi)}\rotimes D_{(\xi)}) \cos(4\pi \sigma_{\xi} {\xi}  \cdot x)\\
		&:= I_1+I_2+I_3+I_4+I_5+I_6
    \end{align*}
    Notice that $\hat{I}_1$ is supported on $\{k \in \mathbb{Z}^3:\sigma<|k|<9\sigma/5\}$, so we can get
    \begin{align*}
        \|I_1\|_{\Hs} 
        &\lesssim \sigma^{-s} \sum_{\xi \in \Lambda_u} \|P_{\le \tau} (a_{(\xi)}^2)\|_{L^\infty} \|P_{\le \tau} (W_{(\xi)}\otimes W_{(\xi)})\|_{L^2} \\
        &\lesssim \sigma^{-s} \sum_{\xi \in \Lambda_u} \|a_{(\xi)}\|_{L^\infty}^2 \|W_{(\xi)}\|_{L^4}^2 \lesssim \ell^{-4\alpha}\sigma^{-s} \rs^{-\frac{1}{2}}.
    \end{align*}
    The estimates of $I_2$ and $I_3$ can be obtained by using Lemma \ref{frecencylemma} as
    \begin{align*}
        \|I_2\|_{L^1} 
        &\leq \sum_{\xi \in \Lambda_u} \|P_{>\tau}(a_{(\xi)}^2)\|_{L^\infty} \|P_{\le \tau} (W_{(\xi)}\otimes W_{(\xi)})\|_{L^1} \\
        &\leq \tau^{-1} \sum_{\xi \in \Lambda_u} \|\nabla (a_{(\xi)}^2)\|_{L^\infty} \|W_{(\xi)}\|_{L^2}^2 \lesssim \ell^{-4\alpha}\tau^{-1},
    \end{align*}
    and
    \begin{align*}
        \|I_3\|_{L^1} 
        &\leq \sum_{\xi \in \Lambda_u} \|a_{(\xi)}^2\|_{L^\infty} \|P_{>\tau} (W_{(\xi)} \otimes W_{(\xi)})\|_{L^1} \\
        &\leq \tau^{-1} \sum_{\xi \in \Lambda_u} \|a_{(\xi)}\|_{L^\infty}^2 \|\nabla ((W_{(\xi)} \otimes W_{(\xi)}))\|_{L^1} \\
        &\leq \tau^{-1} \lambda \sum_{\xi \in \Lambda_u} \|a_{(\xi)}\|_{L^\infty}^2 \|W_{(\xi)}\|_{L^2} \|W_{(\xi)}\|_{L^2} \lesssim\tau^{-1} \ell^{-4\alpha}\lambda.
    \end{align*}
	The estimates for $I_4,I_5,I_6$ are similar to $I_1,I_2,I_3$.
	For $I_4$,
	\begin{align*}
        \|I_4\|_{\Hs} 
        &\leq \sigma^{-s} \sum_{\xi \in \Lambda_B} \|P_{\le \tau} (a_{(\xi)}^2)\|_{L^\infty} \|P_{\le \tau} (W_{(\xi)}\otimes W_{(\xi)} - D_{(\xi)}\otimes D_{(\xi)} )\|_{L^2} \\
        &\leq \sigma^{-s} \sum_{\xi \in \Lambda_B} \|a_{(\xi)}\|_{L^\infty}^2 (\|W_{(\xi)}\|_{L^4}^2 + \|D_{(\xi)}\|_{L^4}^2)\\
		&\lesssim \ell^{-4\alpha}\sigma^{-s} \rs^{-\frac{1}{2}}.
    \end{align*}
    for $I_5$,
    \begin{align*}
        \|I_5\|_{L^1} 
        &\leq \sum_{\xi \in \Lambda_B} \|P_{>\tau}(a_{(\xi)}^2)\|_{L^\infty} \|P_{\le \tau} (W_{(\xi)}\otimes W_{(\xi)} - D_{(\xi)}\otimes D_{(\xi)})\|_{L^1} \\
        &\leq \tau^{-1} \sum_{\xi \in \Lambda_B} \|\nabla (a_{(\xi)}^2)\|_{L^\infty} (\|W_{(\xi)}\|_{L^2}^2 + \|D_{(\xi)}\|_{L^2}^2)\lesssim\ell^{-4\alpha}\tau^{-1},
    \end{align*}
    and for $I_6$,
    \begin{align*}
        \|I_6\|_{L^1} 
        &\leq \sum_{\xi \in \Lambda_B} \|a_{(\xi)}^2\|_{L^\infty} \|P_{>\tau} (W_{(\xi)} \otimes W_{(\xi)}-D_{(\xi)} \otimes D_{(\xi)})\|_{L^1} \\
        &\leq \tau^{-1} \sum_{\xi \in \Lambda_B} \|a_{(\xi)}\|_{L^\infty}^2 \|\nabla (W_{(\xi)} \otimes W_{(\xi)}-D_{(\xi)} \otimes D_{(\xi)})\|_{L^1} \\
        &\leq \tau^{-1} \lambda \sum_{\xi \in \Lambda_B} \|a_{(\xi)}\|_{L^\infty}^2 (\|W_{(\xi)}\|_{L^2} \|W_{(\xi)}\|_{L^2}+\|D_{(\xi)}\|_{L^2} \|D_{(\xi)}\|_{L^2}) \\
        &\lesssim \ell^{-4\alpha}\tau^{-1} \lambda.
    \end{align*}
    Using $L^1$ embedding into $\Hs$ again, we obtain the estimate desired in \eqref{e:Rdis-est}.

    For $\mathring{R}_{{\mathrm{osc}},3}^{u}$, it follows from H\"older's inequality and Lemma \ref{buildingblockestlemma} as
    \begin{align*}
     \|\mathring{R}_{{\mathrm{osc}},3}^{u}\|_{L^1} &\lesssim \|a_{(\xi)}\|_{L^\infty} \|a_{(\xi')}\|_{L^\infty}\max_{\xi \ne \xi'}  (\|W_{(\xi)} \otimes W_{(\xi')}\|_{L^1} +\|D_{(\xi)} \otimes D_{(\xi')}\|_{L^1}) \\
     &\lesssim \ell^{-4\alpha}\rs. 
    \end{align*}

    For $\mathring{R}_{\mathrm{cor}}^{u}$, we can get it from H\"older's inequality directly as
	\begin{align*}
	\|\mathring{R}_{\mathrm{cor}}^{u}\|_{L^1} &\leq 
	\|w_{q+1}^{(p)}\|_{L^1} (\|w_{q+1}^{(c)}\|_{L^\infty} + \|w_{q+1}^{(l)}\|_{L^\infty})\\
	&+\|w_{q+1}\|_{L^1} (\|w_{q+1}^{(c)}\|_{L^\infty} + \|w_{q+1}^{(l)}\|_{L^\infty})\\
	&+\|d_{q+1}^{(p)}\|_{L^1} (\|d_{q+1}^{(c)}\|_{L^\infty} + \|d_{q+1}^{(l)}\|_{L^\infty})\\
	&+\|d_{q+1}\|_{L^1} (\|d_{q+1}^{(c)}\|_{L^\infty} + \|d_{q+1}^{(l)}\|_{L^\infty})\\
	\end{align*}
	So we can get $$ \|\mathring{R}_{\mathrm{cor}}^{u}\|_{L^1} \lesssim \ell^{-4\alpha}\tau^{-1} \lambda+ \ell^{-4\alpha}\sigma^{\delta-1}.$$

    For $\mathring{R}_{\mathrm{lin}}^{u}$, we can first get 
    \begin{align*}
       &\|u_{q} \rotimes w_{q+1} + w_{q+ 1} \rotimes u_{q} - B_{q} \rotimes d_{q+1} - d_{q+1} \rotimes B_{q}\|_{L^1}\\
        &\lesssim \|u_{q}\|_{L^\infty} \|w_{q+1}\|_{L^1} +\|B_{q}\|_{L^\infty}  \|d_{q+1}\|_{L^1}\\
		&\lesssim \ell^{-3\alpha }\rs^{\frac{1}{2}}.
    \end{align*}
    Moreover, as $-s+2\alpha_1-1 < -3/2$, Sobolev's embedding theorem yields that $L^1$ embeds into ${H}^{-s+2\alpha_1-1}$, so we get
    \[ \|\mathcal{R}^u(-\Delta)^{\alpha_1} w_{q+1}\|_{\Hs} \lesssim \|w_{q+1}\|_{{H}^{-s+2\alpha_1-1}} \lesssim \|w_{q+1}\|_{L^1} \le \ell^{-2\alpha}\rs^{\frac{1}{2}}. \]
	This completes the proof.
\end{proof}
Finally, substituting \eqref{choose-prar} into Lemmas \ref{lemma:rb} -- \ref{lemma:ru}, we obtain that
\begin{align}
    \label{e:uBHs}
    \|(\mathring{R}^u_{q+1}, \mathring{R}^B_{q+1})\|_{\Hs} \lesssim  \delta_{q+2}.
\end{align}    
\subsection{$L^{\infty}$-Estimates for the magnetic stress and Reynolds stress.}
\begin{lemma}[$L^{\infty}$-Estimates on $\mathring{R}_{q+1}^B$]
    \label{lemma:rb1}
    Below, we have the following $L^{\infty}$-estimates
    \begin{align}
        \label{b:Rosc-est1}
        &\|\mathring{R}_{\mathrm{osc},1}^{B}\|_{L^{\infty}} \lesssim  \ell^{-2\alpha}\rs^{-1} , \\
        \label{b:Rdis-est1}
        &\|\mathring{R}_{\mathrm{osc},2}^{B}\|_{L^{\infty}} \lesssim  \ell^{-2\alpha}\rs^{-1}, \\
        \label{b:Roff-est1}
        &\|\mathring{R}_{\mathrm{osc},3}^{B}\|_{L^{\infty}} \lesssim \ell^{-3\alpha}\rs^{-1}, \\
        \label{b:Rcor-est1}
        &\|\mathring{R}_{\mathrm{cor}}^{B}\|_{L^{\infty}} \lesssim \ell^{-3\alpha}\tau^{-1} \lambda\rs^{-1}+ \ell^{-3\alpha}\sigma^{\delta -1}\rs^{-1}, \\
        \label{b:Rlin-est1}
        &\|\mathring{R}_{\mathrm{lin}}^{B}\|_{L^{\infty}} \lesssim  \ell^{-\alpha}\rs^{-\frac{1}{2}} \sigma^{2\alpha_2-1+\delta}.
    \end{align}
\end{lemma}

\begin{proof}
    For \eqref{b:Rosc-est1}, by using \eqref{ew-endpt2}, \eqref{ed-endpt2}, \eqref{a-vel-S1-1} and \eqref{a-vel-S1-3}, we consider to calculate $\|\mathring{R}_{{\mathrm{osc}},1}^{B}\|_{L^\infty}$ as
    \begin{align*}
        \|\mathring{R}_{{\mathrm{osc}},1}^{B}\|_{L^{\infty}} 
        &\lesssim  \|a_{(\xi)}^2\|_{L^{\infty}}\|D_{(\xi)}\| _{L^{\infty}}\|W_{(\xi)}\| _{L^{\infty}} \\
        &\lesssim \ell^{-2\alpha}\rs^{-1},
    \end{align*}
    Next, to get \eqref{b:Rdis-est1}, we can calculate as
    \begin{align*}
        \|\mathring{R}_{{\mathrm{osc}},2}^{B}\|_{L^{\infty}} & \lesssim \|a_{(\xi)}^2\|_{L^{\infty}}  \|D_{(\xi)}\|_{L^{\infty}} \| W_{(\xi)} \|_{L^{\infty}} \\
        & \lesssim \ell^{-2\alpha}\rs^{-1}
    \end{align*}
    For $\mathring{R}_{{\mathrm{osc}},3}^{B}$, we can use Lemma \ref{buildingblockestlemma} as
    \[ \|\mathring{R}_{{\mathrm{osc}},3}^{B}\|_{L^{\infty}} \lesssim 
	\max_{\xi\ne \xi'}\|a_{(\xi)}\|_{L^\infty} \|a_{(\xi')}\|_{L^\infty}  \|D_{(\xi')} \|_{L^{\infty}} \| W_{(\xi)}\|_{L^{\infty}} \lesssim \ell^{-3\alpha}\rs^{-1}. \]

    The estimate of $\mathring{R}_{\mathrm{cor}}^{B}$ directly follows from H\"older's inequality.
	\begin{align*}
	\|\mathring{R}_{\mathrm{cor}}^{B}\|_{L^\infty} &\leq
	\|d_{q+1}^{(p)}\|_{L^\infty} (\|w_{q+1}^{(c)}\|_{L^\infty} + \|w_{q+1}^{(l)}\|_{L^\infty})\\
	&+\|d_{q+1}\|_{L^\infty} (\|w_{q+1}^{(c)}\|_{L^\infty} + \|w_{q+1}^{(l)}\|_{L^\infty})\\
	&+\|w_{q+1}^{(p)}\|_{L^\infty} (\|d_{q+1}^{(c)}\|_{L^\infty} + \|d_{q+1}^{(l)}\|_{L^\infty})\\
	&+\|w_{q+1}\|_{L^\infty} (\|d_{q+1}^{(c)}\|_{L^\infty} + \|d_{q+1}^{(l)}\|_{L^\infty}).
	\end{align*}
	So by \eqref{e:wp-est} -- \eqref{e:wc-est} and \eqref{e:dp-est} -- \eqref{e:dc-est}, 
    we can get$$ \|\mathring{R}_{\mathrm{cor}}^{B}\|_{L^\infty} \lesssim \ell^{-3\alpha}\tau^{-1} \lambda\rs^{-1}+ \ell^{-3\alpha}\sigma^{\delta -1}\rs^{-1}.$$

    For $\mathring{R}_{\mathrm{lin}}^{B}$, we can get it by using H\"older's inequality, 
    \begin{align*}
       &\|d_{q + 1} \otimes u_{q} - u_{q} \otimes d_{q+1}+ B_{q}\otimes w_{q+1} -w_{q +1} \otimes B_{q} \|_{L^\infty}\\
        &\lesssim \|B_{q}\|_{L^\infty} \|w_{q+1}\|_{L^\infty} + \|u_{q}\|_{L^\infty}\|d_{q+1}\|_{L^\infty}\\
		&\lesssim \ell^{-3 \alpha}\rs^{-1}.
    \end{align*}
    Moreover, we can get 
    \[ \|\mathcal{R}^B(-\Delta)^{\alpha_2} d_{q+1}\|_{L^\infty} \leq \|d_{q+1}\|_{{C}^{2\alpha_2-1+\delta}}  \lesssim \ell^{-\alpha}\rs^{-\frac{1}{2}} \sigma^{2\alpha_2-1+\delta}. \]
	This completes the proof.
\end{proof}
Similarly, we can get the $L^{\infty}$-estimates of $\mathring{R}_{q+1}^u$.
\begin{lemma}[$L^{\infty}$-Estimates on $\mathring{R}_{q+1}^u$]
    \label{lemma:ru1}
    We have the following estimates for $\mathring{R}_{q+1}^u$, 
    \begin{align}
        \label{e:Rosc-est1}
        &\|\mathring{R}_{\mathrm{osc},1}^{u}\|_{L^{\infty}} \lesssim  \ell^{-4\alpha}\rs^{-1}, \\
        \label{e:Rdis-est1}
        &\|\mathring{R}_{\mathrm{osc},2}^{u}\|_{L^{\infty}} \lesssim   \ell^{-4\alpha}\rs^{-1}, \\
        \label{e:Roff-est1}
        &\|\mathring{R}_{\mathrm{osc},3}^{u}\|_{L^{\infty}} \lesssim \ell^{-4\alpha}\rs^{-1}, \\
        \label{e:Rcor-est1}
        &\|\mathring{R}_{\mathrm{cor}}^{u}\|_{L^{\infty}} \lesssim \ell^{-4\alpha}\tau^{-1} \lambda\rs^{-1}+ \ell^{-8\alpha}\sigma^{\delta-1}\rs^{-1}, \\
        \label{e:Rlin-est1}
        &\|\mathring{R}_{\mathrm{lin}}^{u}\|_{L^{\infty}} \lesssim   \ell^{-2\alpha}\rs^{-\frac{1}{2}}\sigma^{ 2\alpha_1 -1 + \delta}.
    \end{align}
\end{lemma}
Finally, substituting \eqref{choose-prar} into Lemmas \ref{lemma:rb1} -- \ref{lemma:ru1}, we obtain that
\begin{align}
    \label{e:uBL}
    \|(\mathring{R}^u_{q+1}, \mathring{R}^B_{q+1})\|_{L^\infty} \lesssim  \lambda^{12\alpha}.
\end{align}    
\subsection{$C^N$-Estimates for the magnetic stress and Reynolds stress.}
\begin{lemma}[$C^N$-Estimates on $\mathring{R}_{q+1}^B$]
    \label{lemma:rb2}
    For $1 \le N \le 7$, we have the following $C^N$-estimates
    \begin{align}
        \label{b:Rosc-est2}
        &\|\mathring{R}_{\mathrm{osc},1}^{B}\|_{C^N} \lesssim  \ell^{-2\alpha}\lambda^N\rs^{-1} , \\
        \label{b:Rdis-est2}
        &\|\mathring{R}_{\mathrm{osc},2}^{B}\|_{C^N} \lesssim  \ell^{-2\alpha}\sigma^{N}\rs^{-1}, \\
        \label{b:Roff-est2}
        &\|\mathring{R}_{\mathrm{osc},3}^{B}\|_{C^N} \lesssim \ell^{-3\alpha}\sigma^{N}\rs^{-1}, \\
        \label{b:Rcor-est2}
        &\|\mathring{R}_{\mathrm{cor}}^{B}\|_{C^N} \lesssim \ell^{-3\alpha}\sigma^N\rs^{-1}, \\
        \label{b:Rlin-est2}
        &\|\mathring{R}_{\mathrm{lin}}^{B}\|_{C^N} \lesssim \ell^{-\alpha}\rs^{-\frac{1}{2}} \sigma^{2\alpha_2-1+N+\delta}.
    \end{align}
\end{lemma}

\begin{proof}
    For \eqref{b:Rosc-est2}, by \eqref{ew-endpt2}, \eqref{ed-endpt2}, \eqref{a-mag-S1} and \eqref{a-vel-S1-2}, we can calculate $\|\mathring{R}_{{\mathrm{osc}},1}^{B}\|_{C^N}$ as
    \begin{align*}
        \|\mathring{R}_{{\mathrm{osc}},1}^{B}\|_{C^N} 
        &\lesssim  \sum_{N_1 + N_2 + N_3 = N} \|a_{(\xi)}^2\|_{C^{N_1}}\|\nabla ^{N_2} D_{(\xi)}\| _{L^{\infty}}\|\nabla ^{N_3} W_{(\xi)}\|_{L^{\infty}} \\
        &\lesssim \ell^{-2\alpha}\lambda^N \rs^{-1},
    \end{align*}
    Next, to get \eqref{b:Rdis-est2}, we can calculate as
    \begin{align*}
        &\|\mathring{R}_{{\mathrm{osc}},2}^{B}\|_{C^N} \\
        & \lesssim \sum_{N_1 + N_2 + N_3 = N}\|a_{(\xi)}^2\|_{C^{N_1}} \|\nabla ^{N_2} D_{(\xi)}\| _{L^{\infty}}\|\nabla ^{N_3} W_{(\xi)}\|_{L^{\infty}} \| \cos(4\pi \sigma_{\xi} {\xi} \cdot x) \|_{C^{N_4}} \\
        & \lesssim \ell^{-2\alpha}\sigma^{N}\rs^{-1}.
    \end{align*}
    For $\mathring{R}_{{\mathrm{osc}},3}^{B}$, we can use Lemma \ref{buildingblockestlemma} as
    \begin{align*}
    &\|\mathring{R}_{{\mathrm{osc}},3}^{B}\|_{C^N} \\
	& \lesssim \max_{\xi\ne \xi'}\sum_{N_1 + N_2 + N_3 + N_4 + N_5 + N_6= N}\big(\|a_{(\xi)}\|_{C^{N_1}} \|a_{(\xi')}\|_{C^{N_2}}  \|\nabla ^{N_3}D_{(\xi')} \|_{L^{\infty}} \\
    &\| \nabla ^{N_4}W_{(\xi)}\|_{L^{\infty}}\| \cos(2\pi \sigma_{\xi}{\xi} \cdot x)\|_{C^{N_5}} \| \cos(2\pi \sigma_{\xi'}{\xi^{'}} \cdot x)\|_{C^{N_6}}\big)\\
    &\lesssim \ell^{-3\alpha}\sigma^{N}\rs^{-1}. 
    \end{align*}

    The estimate of $\mathring{R}_{\mathrm{cor}}^{B}$ directly follows from H\"older's inequality.
	\begin{align*}
	\|\mathring{R}_{\mathrm{cor}}^{B}\|_{C^N} &\leq
	\sum_{N_1 + N_2 = N}\big(\|d_{q+1}^{(p)}\|_{C^{N_1}} (\|w_{q+1}^{(c)}\|_{C^{N_2}} + \|w_{q+1}^{(l)}\|_{C^{N_2}})\\
	&+\|d_{q+1}\|_{C^{N_1}} (\|w_{q+1}^{(c)}\|_{C^{N_2}} + \|w_{q+1}^{(l)}\|_{C^{N_2}})\\
	&+\|w_{q+1}^{(p)}\|_{C^{N_1}} (\|d_{q+1}^{(c)}\|_{C^{N_2}} + \|d_{q+1}^{(l)}\|_{C^{N_2}})\\
	&+\|w_{q+1}\|_{C^{N_1}} (\|d_{q+1}^{(c)}\|_{C^{N_2}} + \|d_{q+1}^{(l)}\|_{C^{N_2}})\big).
	\end{align*}
	So by \eqref{e:wp-est} -- \eqref{e:wc-est} and \eqref{e:dp-est} -- \eqref{e:dc-est}, 
    we can get$$ \|\mathring{R}_{\mathrm{cor}}^{B}\|_{C^N} \lesssim \ell^{-3\alpha}\sigma^N\rs^{-1}.$$

    For $\mathring{R}_{\mathrm{lin}}^{B}$, by using H\"older's inequality, we first get 
    \begin{align*}
       &\|d_{q + 1} \otimes u_{q} - u_{q} \otimes d_{q+1}+ B_{q}\otimes w_{q+1} -w_{q +1} \otimes B_{q} \|_{C^N}\\
        &\lesssim \sum_{N_1 + N_2 = N} \big(\|B_{q}\|_{C^{N_1}} \|w_{q+1}\|_{C^{N_2}} + \|u_{q}\|_{C^{N_1}}\|d_{q+1}\|_{C^{N_2}}\big)\\
		&\lesssim \ell^{-3 \alpha}\sigma^N\rs^{-1}.
    \end{align*}
    Moreover, we can get 
    \[ \|\mathcal{R}^B(-\Delta)^{\alpha_2} d_{q+1}\|_{C^N} \leq \|d_{q+1}\|_{{C}^{2\alpha_2-1+N+\delta}}  \lesssim \ell^{-\alpha}\rs^{-\frac{1}{2}} \sigma^{2\alpha_2-1+N+\delta}. \]
	This completes the proof.
\end{proof}
Similarly, we can get the $C^N$-estimates of $\mathring{R}_{q+1}^u$.
\begin{lemma}[$C^N$-Estimates on $\mathring{R}_{q+1}^u$]
    \label{lemma:ru2}
    We have the following estimates for $\mathring{R}_{q+1}^u$, 
    \begin{align}
        \label{e:Rosc-est2}
        &\|\mathring{R}_{\mathrm{osc},1}^{u}\|_{C^N} \lesssim  \ell^{-4\alpha}\lambda^N\rs^{-1}, \\
        \label{e:Rdis-est2}
        &\|\mathring{R}_{\mathrm{osc},2}^{u}\|_{C^N} \lesssim   \ell^{-4\alpha}\sigma^{N}\rs^{-1}, \\
        \label{e:Roff-est2}
        &\|\mathring{R}_{\mathrm{osc},3}^{u}\|_{C^N} \lesssim \ell^{-4\alpha}\sigma^{N}\rs^{-1}, \\
        \label{e:Rcor-est2}
        &\|\mathring{R}_{\mathrm{cor}}^{u}\|_{C^N} \lesssim \ell^{-4\alpha}\ell^{-6\alpha}\sigma^N\rs^{-1}, \\
        \label{e:Rlin-est2}
        &\|\mathring{R}_{\mathrm{lin}}^{u}\|_{C^N} \lesssim   \ell^{-2\alpha}\rs^{-\frac{1}{2}}\sigma^{ 2\alpha_1 -1 + N + \delta}.
    \end{align}
\end{lemma}
Finally, substituting \eqref{choose-prar} into Lemmas \ref{lemma:rb2} -- \ref{lemma:ru2}, we obtain that
\begin{align}
    \label{e:uBCN}
    \|(\mathring{R}^u_{q+1}, \mathring{R}^B_{q+1})\|_{C^N} \lesssim  \lambda^{12N\alpha}.
\end{align}    

\subsection{Proof of main iteration Theorem \ref{iteration} }
For (i), let $N_{q+1} = \sigma. $

We can get (i) from Lemma \ref{lemma:w} directly.

For(ii), it is directly from Lemmas \ref{esti-wq} and \ref{esti-dq}.

For(iii), by the definition of $w_{q+1}$ and $d_{q+1}$, we have 
\begin{align*}
    w_{q+1} \otimes w_{q+1} &= w^{(p)}_{q+1} \otimes w^{(p)}_{q+1} + w^{(p)}_{q+1} \otimes (w^{(c)}_{q+1} + w^{(l)}_{q+1} ) 
    +(w^{(c)}_{q+1} + w^{(l)}_{q+1} ) \otimes  w^{(p)}_{q+1} \\
    &+(w^{(c)}_{q+1} + w^{(l)}_{q+1} )  \otimes (w^{(c)}_{q+1} + w^{(l)}_{q+1} ),
\end{align*}
and
\begin{align*}
    d_{q+1} \otimes d_{q+1} &= d^{(p)}_{q+1} \otimes d^{(p)}_{q+1} + d^{(p)}_{q+1} \otimes (d^{(c)}_{q+1} + d^{(l)}_{q+1} ) 
    +(d^{(c)}_{q+1} + d^{(l)}_{q+1} ) \otimes  d^{(p)}_{q+1} \\
    &+(d^{(c)}_{q+1} + d^{(l)}_{q+1} )  \otimes (d^{(c)}_{q+1} + d^{(l)}_{q+1} ).
\end{align*}    
Since 
\begin{align*}
    &w^{(p)}_{q+1} \otimes w^{(p)}_{q+1} \\
    &= \sum_{\xi \in \Lambda_u \cup \Lambda_B} \frac{1}{2} a_{(\xi)}^2 \fint_{\mathbb{T}^3}  W_{(\xi)} \otimes W_{(\xi)}  \d x + \sum_{\xi \in \Lambda_u \cup \Lambda_B} \frac{1}{2} a_{(\xi)}^2 \P_{>0}( W_{(\xi)}\otimes W_{(\xi)}) \\
    &\quad+\sum_{\xi \in \Lambda_u \cup \Lambda_B} \frac{1}{2}  \cos(4\pi \sigma_{\xi}{\xs} \cdot x) a_{(\xi)}^2 ( W_{(\xi)}\otimes W_{(\xi)})\\
    &\quad+ \sum_{\xi \neq \xi' \in \Lambda_u\cup\Lambda_B }a_{(\xi)}a_{(\xi')} \cos(2\pi \sigma_{\xi}{\xs} \cdot x) \cos(2\pi \sigma_{\xi'}{\xi^{'}} \cdot x) W_{(\xi)}\otimes W_{(\xi')}\\
    &= \sum_{\xi \in \Lambda_u \cup \Lambda_B} \frac{1}{2} a_{(\xi)}^2 \fint_{\mathbb{T}^3} ( W_{(\xi)} \otimes W_{(\xi)} ) \d x \\
    &\quad+\frac{1}{2} \sum_{\xi \in \Lambda_u \cup \Lambda_B} \left((P_{\le (\lambda\rs)/2}+ P_{>(\lambda\rs)/2}) (a_{(\xi)}^2)\P_{> 0}(W_{(\xi)}{\otimes}W_{(\xi)}) \right)\\
    &\quad+\frac{1}{2}\sum_{\xi \in \Lambda_u} P_{\le \tau}(a_{(\xi)}^2) P_{\le \tau} (W_{(\xi)}\otimes W_{(\xi)}) \cos(4\pi \sigma_{\xi} {\xi} \cdot x) \\
    &\quad+\frac{1}{2}\sum_{\xi \in \Lambda_u} P_{>\tau}(a_{(\xi)}^2) P_{\le \tau} (W_{(\xi)} \otimes W_{(\xi)}) \cos(4\pi \sigma_{\xi} {\xi} \cdot x) \\   
	&\quad+\frac{1}{2}\sum_{\xi \in \Lambda_u}(a_{(\xi)}^2)  P_{> \tau}(W_{(\xi)}\otimes W_{(\xi)}) \cos(4\pi \sigma_{\xi} {\xi} \cdot x)\\
    &\quad+ \sum_{\xi \neq \xi' \in \Lambda_u\cup\Lambda_B }a_{(\xi)}a_{(\xi')} \cos(2\pi \sigma_{\xi}{\xs} \cdot x) \cos(2\pi \sigma_{\xi'}{\xi^{'}} \cdot x) W_{(\xi)}\otimes W_{(\xi')}
\end{align*}
by Lemmas \ref{Lem-vae-S1} and \ref{esti-wq} and similar estimates as in the proof of Lemma \ref{lemma:ru}, we can obtain that for $s>2\alpha + 3/2$, 
    \begin{align*}
        \|w_{q+1} \otimes w_{q+1} \|_{\Hs} \lesssim \delta_{q+1}.
    \end{align*}
 It is similar with $d_{q+1} \otimes d_{q+1}$ as
    \begin{align*}
         \|d_{q+1} \otimes d_{q+1} \|_{\Hs} \lesssim \delta_{q+1},
    \end{align*}
For $d_{q+1} \otimes w_{q+1}$ and $w_{q+1} \otimes d_{q+1}$, we have 
\begin{align*}
    w_{q+1} \otimes d_{q+1} &= w^{(p)}_{q+1} \otimes d^{(p)}_{q+1} + w^{(p)}_{q+1} \otimes (d^{(c)}_{q+1} + d^{(l)}_{q+1} ) 
    +(w^{(c)}_{q+1} + w^{(l)}_{q+1} ) \otimes  d^{(p)}_{q+1} \\
    &+(w^{(c)}_{q+1} + w^{(l)}_{q+1} )  \otimes (d^{(c)}_{q+1} + d^{(l)}_{q+1} ),
\end{align*}
and
\begin{align*}
    d_{q+1} \otimes w_{q+1} &= d^{(p)}_{q+1} \otimes w^{(p)}_{q+1} + d^{(p)}_{q+1} \otimes (w^{(c)}_{q+1} + w^{(l)}_{q+1} ) 
    +(d^{(c)}_{q+1} + d^{(l)}_{q+1} ) \otimes  w^{(p)}_{q+1} \\
    &+(d^{(c)}_{q+1} + d^{(l)}_{q+1} )  \otimes (w^{(c)}_{q+1} + w^{(l)}_{q+1} ).
\end{align*}   
Since 
\begin{align*}
    &w^{(p)}_{q+1} \otimes d^{(p)}_{q+1} \\
    &= \sum_{\xi \in \Lambda_B} \frac{1}{2} a_{(\xi)}^2 \fint_{\mathbb{T}^3} ( W_{(\xi)} \otimes D_{(\xi)} ) \d x + \sum_{\xi \in  \Lambda_B} \frac{1}{2} a_{(\xi)}^2 \P_{>0}( W_{(\xi)}\otimes D_{(\xi)}) \\
    &\quad+\sum_{\xi \in \Lambda_B} \frac{1}{2}  \cos(4\pi \sigma_{\xi}{\xs} \cdot x) a_{(\xi)}^2 ( W_{(\xi)}\otimes D_{(\xi)})\\
    &\quad+ \sum_{\xi \neq \xi' \in \Lambda_B }a_{(\xi)}a_{(\xi')} \cos(2\pi \sigma_{\xi}{\xs} \cdot x) \cos(2\pi \sigma_{\xi'}{\xi^{'}} \cdot x) W_{(\xi)}\otimes D_{(\xi')}\\
    &= \sum_{\xi \in  \Lambda_B} \frac{1}{2} a_{(\xi)}^2 \fint_{\mathbb{T}^3} ( W_{(\xi)} \otimes D_{(\xi)} ) \d x \\
    &\quad+\frac{1}{2} \sum_{\xi \in  \Lambda_B} \left((P_{\le (\lambda\rs)/2}+ P_{>(\lambda\rs)/2}) (a_{(\xi)}^2)\P_{> 0}(W_{(\xi)}{\otimes}D_{(\xi)}) \right)\\
    &\quad+\frac{1}{2}\sum_{\xi \in \Lambda_B} P_{\le \tau}(a_{(\xi)}^2) P_{\le \tau} (W_{(\xi)}\otimes D_{(\xi)}) \cos(4\pi \sigma_{\xi} {\xi} \cdot x) \\
    &\quad+\frac{1}{2}\sum_{\xi \in \Lambda_B} P_{>\tau}(a_{(\xi)}^2) P_{\le \tau} (W_{(\xi)} \otimes D_{(\xi)}) \cos(4\pi \sigma_{\xi} {\xi} \cdot x) \\   
	&\quad+\frac{1}{2}\sum_{\xi \in \Lambda_B}(a_{(\xi)}^2)  P_{> \tau}(W_{(\xi)}\otimes D_{(\xi)}) \cos(4\pi \sigma_{\xi} {\xi} \cdot x)\\
    &\quad+ \sum_{\xi \neq \xi' \in \Lambda_B }a_{(\xi)}a_{(\xi')} \cos(2\pi \sigma_{\xi}{\xs} \cdot x) \cos(2\pi \sigma_{\xi'}{\xi^{'}} \cdot x) W_{(\xi)}\otimes D_{(\xi')},
\end{align*}
by Lemmas \ref{Lem-mae-S1} and \ref{esti-dq} and similar estimates as in the proof of Lemma \ref{lemma:rb}, we can obtain that for $s>3/2 + 2 \alpha$, 
    \begin{align*}
        &\|w_{q+1} \otimes d_{q+1} \|_{\Hs} \lesssim \delta_{q+1}.
    \end{align*}    
Estimates on $\|d_{q+1} \otimes w_{q+1} \|_{\Hs}$ can be obtained by exchanging $D_{(\xi)}$ and $W_{(\xi)}$.
    
    Next, we use H\"older's inequality and  $L^1$ embedding into $\Hs$ such that
     \begin{align*} 
        &\sum_{M \le 2N} \big( \|P_M u_q \otimes w_{q+1} \|_{\Hs} + \|w_{q+1} \otimes P_M u_q \|_{\Hs} )\\
        &\lesssim \sum_{M \le 2N} \big( \| P_M u_q \otimes w_{q+1} \|_{L^1} + \|w_{q+1} \otimes P_M u_q\|_{L^1} ) \\
        &\lesssim  \|u_q\|_{B^0_{\infty,1}} \|w_{q+1}\|_{L^1} \lesssim \ell^{-3 \alpha} r < \delta_{q+2}.
     \end{align*}
   
    By using H\"older's inequality, we have 
     \begin{align*}
        &\sum_{M \le 2N} \big( \|P_M B_q \otimes d_{q+1} \|_{\Hs} + \| d_{q+1} \otimes P_M B_q\|_{\Hs} )\\
        &\lesssim \sum_{M \le 2N} \big(\|P_M B_q \otimes d_{q+1}\|_{L^1} + \| d_{q+1} \otimes P_M B_q\|_{L^1})\\
        &\lesssim  \|B_q\|_{B^0_{\infty,1}} \|d_{q+1}\|_{L^1}  \lesssim \ell^{-3 \alpha} r < \delta_{q+2}, 
     \end{align*}
      and 
      \begin{align*}
        &\sum_{M \le 2N} \big(\|w_{q+1} \otimes P_M B_q \|_{\Hs} + \|  P_M u_q \otimes d_{q+1} \|_{\Hs} )\\
        &\lesssim \sum_{M \le 2N} \big(\|w_{q+1} \otimes P_M B_q\|_{L^1} + \|  P_M u_q \otimes d_{q+1} \|_{L^1})\\
        &\lesssim  \|B_q\|_{B^0_{\infty,1}} \|w_{q+1}\|_{L^1} +  \|u_q\|_{B^0_{\infty,1}} \|d_{q+1}\|_{L^1} \lesssim \ell^{-3\alpha} r < \delta_{q+2}, 
     \end{align*}
     moreover,
     \begin{align*}
        &\sum_{M \le 2N} \big(\|d_{q+1} \otimes P_M u_q \|_{\Hs} + \|  P_M B_q \otimes u_{q+1} \|_{\Hs} )\\
        &\lesssim \sum_{M \le 2N} \big(\|d_{q+1} \otimes P_M u_q\|_{L^1} + \|  P_M B_q \otimes u_{q+1} \|_{L^1})\\
        &\lesssim  \|u_q\|_{B^0_{\infty,1}} \|d_{q+1}\|_{L^1} +  \|B_q\|_{B^0_{\infty,1}} \|u_{q+1}\|_{L^1} \lesssim \ell^{-3\alpha} r < \delta_{q+2}, 
     \end{align*}
    so we get(iii).

    Next, we check the iteration estimates \eqref{mitr1} -- \eqref{mitr2}. 
    For \eqref{mitr1}, by \eqref{mitr1_}, \eqref{e:dw-est-besov} and \eqref{e:dd-est-besov} we have 
\begin{align*}
	& \|(u_{q+1}, B_{q+1})\|_{L^\infty} \\
	& \leq \|(u_{q}, B_{q})\|_{L^\infty} + \|(w_{q+1}, d_{q+1})\|_{L^\infty}\\
	& \lesssim \lambda_{q}^{2}+ \ell^{-2}\rs^{-\frac{1}{2}}\\
	& \lesssim \lambda_{q+1}^{2}
\end{align*}
Then, we can calculate \eqref{mitr2} by \eqref{mitr2_}, \eqref{e:dw-est-besov} and \eqref{e:dd-est-besov}
\begin{align*}
	& \|(u_{q+1}, B_{q+1})\|_{C^N} \\
	& \leq \|(u_{q}, B_{q})\|_{C^N} + \|(w_{q+1}, d_{q+1})\|_{C^N}\\
	& \lesssim \lambda_{q}^{6N}+ \ell^{-2 \alpha}\lambda_{q+1}^{5N}\rs^{-\frac{1}{2}}\\
	& \lesssim \lambda_{q+1}^{6N}
\end{align*}
For \eqref{mitr6}, we have 
\begin{align*}
	& \|(u_{q+1}, B_{q+1})\|_{B^0_{\infty,1}} \\
	& \leq \|(u_{q}, B_{q})\|_{B^0_{\infty,1}} + \|(w_{q+1}, d_{q+1})\|_{B^0_{\infty,1}}\\
	& \lesssim \lambda_{q}^{6}+ \ell^{-4\alpha} \sigma ^{\delta} \rs^{-\frac{1}{2}}\\
	& \lesssim \lambda_{q+1}^{6}
\end{align*}

    Finally, the estimates \eqref{mitr3} -- \eqref{mitr5} directly follow \eqref{e:uBHs}, \eqref{e:uBL}, \eqref{e:uBCN}.

\section{Proof of Theorem \ref{thm2}}

In this section, we would use the main iteration Theorem \ref{iteration} to prove Theorem \ref{thm2} by three steps.\\
{\bf Step 1: verifying the start step in the iteration.}
First, we verify the iterative step for $q=0$. 
We select a smooth pair $(u_0, B_0)$ satisfying $\operatorname{div} u_0 = \operatorname{div} B_0 = 0$
with $\supp \hat{u}_0 \subset \{k\in \mathbb{Z}^3:|k|=1\}$ and $\hat{B}_0 (0)=B_{*}$, $\hat{B}_{0} (\xi)=0$, $\forall \xi \ne 0$,
for instance a sequence of shear flows 
$
 u^{(m)}_0 = (cos(2\pi x_3),0,0) 
 $
 with 
 $
 B^{(m)}_0 = B_{*}.
$  

Define:
\begin{align}
	&\mathring{R}_0^u=\mathcal{R}^u\(\nu_1(-\Delta)^{\alpha} u_0\) + u_0\mathring\otimes u_0-B_0\mathring\otimes B_0, \label{r0u}\\
	& \mathring{R}_0^B=\mathcal{R}^B\(\nu_2(-\Delta)^{\alpha} B_0\) + B_0\otimes u_0-u_0\otimes B_0,   \label{r0b}
\end{align}
together with
$ P_0 = -\frac{1}{3} (|u_0|^2 - |B_0|^2)$.

Let
\[ w_0:=u_0, \quad N_0:=\max\{2,B_{*}\}, \]
we can get $(u_0,B_0,\mathring{R}^{u}_0, \mathring{R}^{B}_0)$ is a smooth solution to \eqref{R_MHDs} with $\supp \hat{u}_0 \subset B(0,N_0)$ and $\supp \hat{B}_0 \subset B(0,N_0)$. 
Then, applying Theorem \ref{iteration}, we get a sequence of approximate solutions
$ (u_q, B_q, \mathring{R}^{u}_q, \mathring{R}^{B}_q)$ to \eqref{R_MHDs} satisfying the estimates \eqref{mitr1_} -- \eqref{mitr3_} and \ref{item:prop-hatw-supp} -- \ref{item:prop-w-paraproduct}. 
It follows that $(u_q, B_q) $ is a Cauchy sequence in $B^{-\theta}_{\infty,1}$, and $(\mathring{R}^{u}_q, \mathring{R}^{B}_q)\to 0$ \text{in} $ H^{-s} $. We define
$$u := \lim_{q \to \infty} u_q, \quad B: = \lim_{q \to \infty} B_q \quad \text{in } {B}^{-\theta}_{q,r}(\mathbb{T}^3).$$
Let 
$
\displaystyle\sum_{q\ge 1}\delta_q \ll \|u_0\|_{B^{-\theta}_{\infty,1}},
$
since  
\begin{align*}
\| u \|_{B^{-\theta}_{\infty,1}} \ge  \| u_0 \|_{B^{-\theta}_{\infty,1}} -\| u - u_0 \|_{B^{-\theta}_{\infty,1}} >  \frac{1}{2} \| u_0 \|_{B^{-\theta}_{\infty,1}} ,
\end{align*}
we can get a sequence of different limit functions.\\
{\bf Step 2: verifying the tensor products to be well-defined.}
Next, we verify that $u\otimes u, u\otimes B,$ $B\otimes u$ \text{and} $B\otimes B$ are well-defined in $H^{-s}$ and are the limits of $u_q\otimes u_q, u_q\otimes B_q,$ $B_q\otimes u_q$ \text{and} $B_q\otimes B_q$.
Since $u\otimes u, B\otimes B, u\otimes B,$ \text{and} $B\otimes u $ have similar structures,
we would only present the calculations for $u\otimes u$ \text{and} $u\otimes B$ and the other cases follow similarly.
Applying Theorem \ref{iteration}, we check $u\otimes u$ as
\begin{align*}
    \sum_{1\le Q,Q' \le Q_{\mathrm{max}}} \|P_Q u \otimes P_{Q^{'}}u\|_{\Hs}& \leq \sum_{2Q_{\mathrm{max}} \ge q,q' \ge 0} \|w_q \otimes w_{q'}\|_{\Hs} \\
    &= \|w_0 \otimes w_0\|_{\Hs} + \sum_{2Q_{\mathrm{max}} \ge q \ge 1} \|w_q \otimes w_q\|_{\Hs} \\
    &+ \sum_{1\le q' < q \le 2Q_{\mathrm{max}}} \big(\|w_q \otimes w_{q'}\|_{\Hs} + \|w_{q'} \otimes w_q\|_{\Hs} \big) \\
    &\lesssim \|w_0 \otimes w_0\|_{\Hs} + \sum_{q\ge 1} \delta_{q}   < \infty.
\end{align*}
Similarly, we check $u\otimes B$ as
\begin{align*}
    \sum_{1\le Q,Q' \le Q_{\mathrm{max}}} \|P_Q u \otimes P_{Q^{'}}B\|_{\Hs} & \leq \sum_{2Q_{\mathrm{max}} \ge q,q' \ge 0} \|w_q \otimes d_{q'}\|_{\Hs} \\
    &= \|w_0 \otimes d_0\|_{\Hs} + \sum_{2Q_{\mathrm{max}} \ge q \ge 1} \|w_q \otimes d_q\|_{\Hs} \\
    &\quad + \sum_{1\le q' < q \le 2Q_{\mathrm{max}}} \big(\|w_q \otimes d_{q'}\|_{\Hs} + \|w_{q'} \otimes d_q\|_{\Hs} \big) \\
    &\lesssim \|w_0 \otimes d_0\|_{\Hs} + \sum_{q\ge 1} \delta_{q} < \infty.
\end{align*}
Finally, we check the limits as 
\begin{align*}
\lim_{q\to\infty} u_q\otimes u_q = \lim_{q\to\infty} P_{\leq N_q}u \otimes P_{\leq N_q}u = u\otimes u \quad\text{in } \Hs,\\
\lim_{q\to\infty} u_q\otimes B_q = \lim_{q\to\infty} P_{\leq N_q}u \otimes P_{\leq N_q}B = u\otimes B \quad\text{in } \Hs,\\
\lim_{q\to\infty} B_q\otimes u_q = \lim_{q\to\infty} P_{\leq N_q}B \otimes P_{\leq N_q}u = B\otimes u \quad\text{in } \Hs,\\
\lim_{q\to\infty} B_q\otimes B_q = \lim_{q\to\infty} P_{\leq N_q}u \otimes P_{\leq N_q}u = B\otimes B \quad\text{in } \Hs.
\end{align*}
{\bf Step 3: verifying the limits $(u,B)$ as solutions to \eqref{equa-sMHD}.}
Finally, we verify that $(u, B) $ satisfies the equations \eqref{equa-sMHD}.
For any divergence-free test functions $\varphi \in C^\infty(\mathbb{T}^3)$,
\[
-\langle u_q \otimes u_q - B_q \otimes B_q, \nabla\varphi\rangle + \langle u_q,  \nu_1(-\Delta)^{\alpha_1}\varphi\rangle = \langle \mathring{R}^{u}_q, - \nabla\varphi\rangle,
\]
	\[
	-\langle B_q \otimes u_q - u_q \otimes B_q, \nabla\varphi\rangle + \langle B_q,  \nu_2(-\Delta)^{\alpha_2}\varphi\rangle = \langle \mathring{R}^{B}_q, - \nabla\varphi\rangle,
	\]
Since $\|(\mathring{R}^u_q ,\mathring{R}^B_q  )\|_{\Hs} < \delta_{q+1},$ we can get that $(u,B)$ is a solution to \eqref{equa-sMHD} by taking limits $q\to \infty$.

Moreover, we show the way to get infinite solutions to \eqref{equa-sMHD}. 
Let $\{u_{0}^{(i)},B_{0}^{(i)}\}$ are a sequence of smooth divergence-free vector fields, where $B_{0}^{(i)}$ always satisfies $B_{0}^{(i)}(0) = B_*$ and $B_{0}^{(i)}(\xi) = 0,\ \forall \xi \ne 0$. 
We pick 
\[
\|u_{0}^{(i)} - u_{0}^{(j)}\|_{{B}^{-\theta}_{q,r}} > 3,\quad \forall i \ne j,
\]
Since 
\[
\|(u^{(i)},B^{(i)}) -(u_0,B_0)\|_{{B}^{-\theta}_{q,r}} \le \sum_{q=0}^{\infty} \| (w_{q+1} , d_{q+1}) \|_{{B}^{-\theta}_{q,r}}\le \sum_{q=0}^{\infty} \delta_{q+2}  \le 1,
\]
we can get 
\[
\|(u^{(i)},B^{(i)}) - (u^{(j)},B^{(j)})\|_{{B}^{-\theta}_{q,r}} 
\ge 1, \quad \forall i \ne j,
\]
where $(u^{(i)},B^{(i)})$ denote the singular weak solution with $(u_{0}^{(i)},B_{0}^{(i)})$ starting the iteration. 
Thus for any $i \ne j$, we have $(u^{(i)},B^{(i)}) \ne (u^{(j)},B^{(j)})$, which completes the proof.
\section*{Acknowledgments}
This research is supported in part by NSFC Grants No.12431007 and 62588101.

\bibliographystyle{plain}  
\bibliography{sMHD_ref}  

@article{NA25,
	author = {N. Gismondi and A. F. Radu},
	journal = {arXiv:2510.16583},
	title = {Intermittent solutions of the stationary 2D surface quasi-geostrophic equation},
	year = {2025}}

@article{CS05,
	author = {M. Christ},
	journal = {arXiv preprint math/0503366},
	title = {Nonuniqueness of weak solutions of the nonlinear {S}chr{\"o}dinger equation},
	year = {2005}}

@article{ABGM25,
	author = {E. Ashkarian and A. Bhargava and N. Gismondi and M. Novack},
	journal = {arXiv preprint arXiv:2506.00841},
	title = {Intermittent singular solutions of the stationary {2D} {N}avier-{S}tokes equations in sharp {S}obolev spaces},
	year = {2025}}

@article{CM05,
	author = {M. Christ},
	journal = {Preprint available https://math. berkeley. edu/mchrist/Papers/stokes. pdf},
	title = {NONUNIQUENESS OF GENERALIZED SOLUTIONS OF THE {N}AVIER-{S}TOKES EQUATION},
	year = {2005}}

@article{DS13,
	author = {C. {De Lellis} and L. Sz\'{e}kelyhidi, Jr.},
	journal = {Invent. Math.},
	number = {2},
	pages = {377--407},
	title = {Dissipative continuous {E}uler flows},
	volume = {193},
	year = {2013}}

@article{NY2025,
	author = {Y. Nie and W. Ye},
	journal = {Journal of Nonlinear Science},
	number = {5},
	pages = {102},
	title = {Sharp and Strong Nonuniqueness for the Magnetohydrodynamic Equations},
	volume = {35},
	year = {2025}}

@article{MY24,
	author = {C. Miao and W. Ye},
	journal = {J. Math. Pures. Appl.},
	pages = {190-227},
	title = {On the weak solutions for the {MHD} systems with controllable total energy and cross helicity},
	volume = {181},
	year = {2024}}

@article{MNY25,
	author = {C. Miao and Y. Nie and W. Ye},
	journal = {Ann. PDE},
	number = {2},
	title = {On {O}nsager-Type Conjecture for the {E}ls{\"a}sser Energies of the Ideal {MHD} Equations},
	volume = {11},
	year = {2025}}

@article{Nash1954,
	author = {J. Nash.},
	journal = {Ann. of Math. (2)},
	pages = {383--396},
	title = {{$C^1$} isometric imbeddings},
	volume = {60},
	year = {1954}}

@article{MI26,
	author = {M. Fujii},
	journal = {arXiv:2602.19846},
	title = {Sharp non-uniqueness for the {Navier--Stokes} equations in scaling critical spaces},
	year = {2026}}

@article{MS26,
	author = {M. P. Coiculescu and S. Palasek},
	journal = {Invent. Math.},
	number = {1},
	pages = {165-219},
	title = {Non-uniqueness of smooth solutions of the {Navier--Stokes} equations from critical data},
	volume = {244},
	year = {2026}}

@article{CH26,
	author = {A. Cheskidov and H. Hou},
	journal = {arXiv:2603.03666},
	title = {On non-uniqueness of mild solutions and stationary singular solutions to the {Navier-Stokes} equations},
	year = {2026}}

@article{nv23,
	author = {M. Novack and V. Vlad},
	journal = {Invent. Math.},
	number = {1},
	pages = {223--323},
	title = {An intermittent {O}nsager theorem},
	volume = {233},
	year = {2023}}

@article{cl22,
	author = {A. Cheskidov and X. Luo},
	journal = {Anal. \& PDE},
	number = {6},
	pages = {2161-2177},
	title = {Extreme temporal intermittency in the linear {Sobolev} transport: almost smooth nonunique solutions},
	volume = {17},
	year = {2024}}

@article{bbv20,
	author = {R. Beekie and T. Buckmaster and V. Vicol},
	journal = {Ann. PDE},
	number = {1},
	pages = {Paper No. 1, 40pp},
	title = {Weak solutions of ideal {MHD} which do not conserve magnetic helicity},
	volume = {6},
	year = {2020}}

@article{bdis15,
	author = {T. Buckmaster and C. De Lellis and P. Isett and L. Sz\'{e}kelyhidi, Jr.},
	journal = {Ann. of Math. (2)},
	number = {1},
	pages = {127--172},
	title = {Anomalous dissipation for {$1/5$}-{H}\"{o}lder {Euler} flows},
	volume = {182},
	year = {2015}}

@article{bv19b,
	author = {T. Buckmaster and V. Vicol},
	journal = {Ann. of Math. (2)},
	number = {1},
	pages = {101--144},
	title = {Nonuniqueness of weak solutions to the {Navier-Stokes} equation},
	volume = {189},
	year = {2019}}

@article{CKS97,
	author = {R. E. Caflisch and I. Klapper and G. Steele},
	journal = {Comm. Math. Phys.},
	number = {2},
	pages = {443--455},
	title = {Remarks on singularities, dimension and energy dissipation for ideal hydrodynamics and {MHD}},
	volume = {184},
	year = {1997}}

@article{cl21,
	author = {A. Cheskidov and X. Luo},
	journal = {Ann. PDE},
	number = {1},
	pages = {1-45},
	title = {Nonuniqueness of weak solutions for the transport equation at critical space regularity},
	volume = {7},
	year = {2021}}

@article{EPT26,
	author = {A. Enciso and J. {Pe{\~n}afiel-Tom{\'a}s} and D. Peralta-Salas},
	journal = {arXiv:2507.23749},
	title = {H\"older continuous dissipative solutions of ideal {MHD} with nonzero helicity},
	year = {2025}}

@article{cl20.2,
	author = {A. Cheskidov and X. Luo},
	journal = {Invent. math.},
	number = {3},
	pages = {987-1054},
	title = {Sharp nonuniqueness for the {Navier-Stokes} equations},
	volume = {229},
	year = {2022}}

@article{dai18,
	author = {M. Dai},
	journal = {SIAM J. Math. Anal.},
	number = {5},
	pages = {5979--6016},
	title = {Non-uniqueness of {Leray-Hopf} weak solutions of the 3d {Hall-MHD} system},
	volume = {53},
	year = {2021}}

@article{dls09,
	author = {C. {De Lellis} and L. Sz\'{e}kelyhidi, Jr.},
	journal = {Ann. of Math. (2)},
	number = {3},
	pages = {1417--1436},
	title = {The {Euler} equations as a differential inclusion},
	volume = {170},
	year = {2009}}

@article{fl20,
	author = {D. Faraco and S. Lindberg},
	journal = {Comm. Math. Phys.},
	number = {2},
	pages = {707--738},
	title = {Proof of {Taylor's} conjecture on magnetic helicity conservation},
	volume = {373},
	year = {2020}}

@article{FL22,
	author = {D. Faraco and S. Lindberg},
	journal = {Geophysical \& Astrophysical Fluid Dynamics},
	number = {4},
	pages = {25pp},
	title = {Rigorous results on conserved and dissipated quantities in ideal {MHD} turbulence},
	volume = {116},
	year = {2022}}

@article{wu03,
	author = {J. Wu},
	journal = {J. Differential Equations},
	number = {2},
	pages = {284--312},
	title = {Generalized {MHD} equations},
	volume = {195},
	year = {2003}}

@article{I18,
	author = {P. Isett},
	journal = {Ann. of Math. (2)},
	number = {3},
	pages = {871--963},
	title = {A proof of {Onsager's} conjecture},
	volume = {188},
	year = {2018}}

@article{KL07,
	author = {E. Kang and J. Lee},
	journal = {Nonlinearity},
	number = {11},
	pages = {2681--2689},
	title = {Remarks on the magnetic helicity and energy conservation for ideal magneto-hydrodynamics},
	volume = {20},
	year = {2007}}

@article{lqzz22,
	author = {Y. Li and Z. Zeng and D. Zhang},
	journal = {J. Math. Pure. Appl.},
	pages = {232-285},
	title = {Non-uniqueness of weak solutions to 3{D} magnetohydrodynamic equations},
	volume = {165},
	year = {2022}}

@article{lqzz24,
	author = {Y. Li and P. Qu and Z. Zeng and D. Zhang},
	journal = {J. Math. Pures Appl.(9)},
	pages = {Paper No. 103602, 64pp},
	title = {Sharp non-uniqueness for the 3{D} hyperdissipative {Navier-Stokes} equations: beyond the {Lions} exponent},
	volume = {190},
	year = {2024}}

@article{lq20,
	author = {T. Luo and P. Qu},
	journal = {J. Differential Equations},
	number = {4},
	pages = {2896--2919},
	title = {Non-uniqueness of weak solutions to 2{D} hypoviscous {Navier-Stokes} equations},
	volume = {269},
	year = {2020}}

@article{lt20,
	author = {T. Luo and E. S. Titi},
	journal = {Calc. Var. Partial Differential Equations},
	number = {3},
	pages = {Paper No. 92, 15pp},
	title = {Non-uniqueness of weak solutions to hyperviscous {Navier-Stokes} equations: on sharpness of {J.-L. Lions} exponent},
	volume = {59},
	year = {2020}}

@article{luo19,
	author = {X. Luo},
	journal = {Arch. Ration. Mech. Anal.},
	number = {2},
	pages = {701--747},
	title = {Stationary solutions and nonuniqueness of weak solutions for the {Navier-Stokes} equations in high dimensions},
	volume = {233},
	year = {2019}}

@article{ms18,
	author = {S. Modena and L. Sz\'{e}kelyhidi, Jr.},
	journal = {Ann. PDE},
	number = {2},
	pages = {Paper No. 18, 38pp},
	title = {Non-uniqueness for the transport equation with {Sobolev} vector fields},
	volume = {4},
	year = {2018}}

@article{ST83,
	author = {M. Sermange and R. Temam},
	journal = {Comm. Pure Appl. Math.},
	number = {5},
	pages = {635--664},
	title = {Some mathematical questions related to the {MHD} equations},
	volume = {36},
	year = {1983}}

@article{taylor86,
	author = {J. B. Taylor},
	journal = {Reviews of Modern Physics},
	number = {3},
	pages = {741},
	title = {Relaxation and magnetic reconnection in plasmas},
	volume = {58},
	year = {1986}}

\end{document}